\documentclass[journal]{IEEEtran}
\ifCLASSINFOpdf
\else
\fi

\usepackage{booktabs} 
\usepackage{amsfonts}
\usepackage{amsthm}
\usepackage{mathrsfs} 
\usepackage{amsmath} 
\usepackage{psfrag} 
\usepackage{array,booktabs,calc}
\usepackage{enumerate}
\usepackage{bbm}
\usepackage{graphicx}
\usepackage{color}
\usepackage{makecell}
\usepackage[capitalise]{cleveref}
\usepackage{mathtools}

\usepackage{enumitem}
\def\mathclap#1{\text{\hbox to 0pt{\hss$\mathsurround=0pt#1$\hss}}}

\newtheorem{assumption}{Assumption}

\newtheorem{definition}{Definition}
\newtheorem{remark}{Remark}
\newtheorem{corollary}{Corollary}
\newtheorem{lemma}{Lemma}
\newtheorem{theorem}{Theorem}

\newcommand{\X}{\mathcal{X}}

\newcommand{\set}{\mathcal{S}}
\newcommand{\setU}{\mathcal{U}}
\newcommand{\setI}{\mathcal{I}}
\newcommand{\XR}{\hat{\mathcal{X}}}

\newcommand{\Real}{\mathbb{R}}
\newcommand{\Complex}{\mathbb{C}}
\newcommand{\iu}{{\bf i}}
\newcommand{\K}{\mathbb{K}}
\newcommand{\Zahlen}{\mathbb{Z}}

\newcommand{\xlo}{x^{\rm lo}}
\newcommand{\tlim}{t^{\dagger}}
\newcommand{\Hn}{{\mathsf H}}
\newcommand{\co}{{\mathsf c}}
\newcommand{\tr}{{\rm tr}}
\newcommand{\rank}{{\rm rank}}

\newcommand{\G}{\mathcal{G}}
\newcommand{\V}{\mathcal{V}}
\newcommand{\E}{\mathcal{E}}
\newcommand{\M}{\mathcal{M}}

\newcommand{\hs}{\tilde{s}}
\newcommand{\hv}{\tilde{v}}
\newcommand{\hl}{\tilde{\ell}}
\newcommand{\hS}{\tilde{S}}

\newcommand{\bh}{\bar{h}}
\newcommand{\bhs}{\overline{h^*}}
\newcommand{\hfea}{h_{\rm fea}}
\newcommand{\Lfea}{L_{\rm fea}}
\newcommand{\Hset}{\mathcal{H}}
\newcommand{\bHset}{\mathcal{\bar{H}}}
\newcommand{\tHset}{\mathcal{\hat{H}}}
\newcommand{\ball}{\mathcal{B}}
\newcommand{\td}{t^\dag}
\newcommand{\tdd}{t^\ddag}
\newcommand{\deltap}{\delta_+}
\newcommand{\deltam}{\delta_-}

\DeclareMathOperator*{\argmin}{arg\,min}
\DeclareMathOperator*{\argmax}{arg\,max}

\newcommand{\Vd}{V^\dagger}
\newcommand{\hd}{h^\dagger}
\newcommand{\hdd}{h^\ddagger}

\newcommand{\hxtra}{h^{\wr}}

\newcommand{\tV}{\tilde{V}}
\newcommand{\tpath}{\tilde{h}}

\newcommand{\proj}{{\mathcal P}}

\begin{document}
\title{\huge Conditions for Exact Convex Relaxation and No Spurious Local Optima}

\author{Fengyu~Zhou,~\IEEEmembership{Student Member,~IEEE,}
        and~Steven~H.~Low,~\IEEEmembership{Fellow,~IEEE}
\thanks{
Partial and preliminary results have appeared in \cite{ZhoL20}.

Fengyu Zhou and Steven H. Low are with the Department
of Electrical Engineering, California Institute of Technology, Pasadena,
CA, 91125 USA e-mail: (\{f.zhou, slow\}@caltech.edu).}}

%



\maketitle

\begin{abstract}
Non-convex optimization problems can be approximately solved via relaxation or local algorithms.
For many practical problems such as optimal power flow (OPF) problems, both approaches tend to succeed in the sense that 
relaxation is usually exact and local algorithms usually converge to a global optimum.
In this paper, we study conditions which are sufficient or necessary for such non-convex problems to simultaneously have exact relaxation and no spurious local optima. 
Those conditions help us explain the widespread empirical experience that local algorithms for OPF problems often work extremely well.
\end{abstract}


\IEEEpeerreviewmaketitle

\section{Introduction}
\IEEEPARstart{N}{on-convex} optimization problems in general are computationally challenging.
However, many heuristics tend to work well for real-world problems.
Those approaches include convex relaxations and local algorithms.
It is usually hoped that relaxations yields exact solutions and local optima are also globally optimal.
In this paper, we derive
the conditions, sufficient or necessary, for these two properties to hold simultaneously.
Our focus is specifically on the optimization formulation with convex cost and non-convex constraints.

\subsection{Related Works}
Many problems have been proved to have exact relaxation and no spurious local optima 
(such as matrix completion \cite{candes2009exact, candes2010power, ge2016matrix}, low rank semidefinite program \cite{barvinok1995problems, pataki1998rank,burer2005local}),
the proofs for those two properties are usually based on different types of certificates.
In this subsection, we review some widely-used certificates for each property.

One type of certificates to exhibit relaxation exactness is by showing that any relaxed (and infeasible) point maps to a feasible solution with lower cost.
This asserts that relaxed points cannot be the optimal solution.
For instance, \cite{farivar2013branchI,gan2015exact} prove that optimal power flow problems can be solved via second-order cone relaxation under certain conditions,
using the argument that any solution in the interior of the second-order cone can always be moved towards the boundary to further reduce the cost.
In \cite{barvinok1995problems, pataki1998rank}, it is proved that if a semi-definite program has a solution with sufficiently large rank, then one can always reduce the rank without increasing the cost or violating the constraints.
Another type of certificates involve studying the dual variables and KKT conditions.
The underlying idea is a pair of primal and dual solutions satisfying KKT conditions that certify their optimality for both the primal and dual problems.
Thus constructing dual variables with certain structures can also certify the optimality of primal solutions.
In \cite{candes2009exact} for instance, the dual variable is related to the subgradient of the cost function at a desired matrix and therefore it helps certify the optimality of that desired matrix.
Another example is \cite{lavaei2012zero}, which proves the primal matrix should be of rank $1$ through the argument that the null space of its dual matrix has dimension at most $1$.
Similar techniques are also used in \cite{jalden2003semidefinite, lu2019tightness, li2015sufficient}.

There are also considerable literature establishing the global optimality of local optima.
We refer to \cite{sun2015nonconvex, ge2017no} and references therein.
In \cite{sun2015nonconvex}, the authors focus on a class of problems with twice continuously differentiable function as the cost and Riemannian manifold as the feasible set.
The values of Riemannian gradient and Hessian at a certain point then help certify properties such as {\it strong gradient}, {\it negative curvature} or {\it local convexity} in its neighborhood.
It then eliminates spurious local optima and saddle points, where local algorithms can be trapped.
This technique or idea were also used in \cite{sun2016complete} for dictionary recovery problem and \cite{boumal2016nonconvex} for phase synchronization.
In both problems, the Riemannian manifold is some $n$-sphere or the Cartesian product of $n$-spheres.
The framework summarized in \cite{ge2017no} also leverages the landscape of the cost function, and the problem is usually reformulated into an unconstrained form.
Instead of explicitly computing the gradient and Hessian matrix, the paper shows it suffices to find a single direction of improvement.
For certain symmetric positive definite problems, the paper shows that the decision variable will always get closer to the global optimizer when the cost is reduced.
A similar idea was also applied in \cite{arora2015simple}, 
where the main result is built upon a correlation condition which states that the gradient (or any updating rule) is correlated with the direction from the current location towards the global optimizer.
Therefore the underlying algorithm such as gradient descent can always produce a solution closer to the global optimizer as the algorithm progresses.

\subsection{Contribution}
The brief review above shows most works study exact relaxation and local optimality separately.
It is unclear what might be the common feature of non-convex problems that possess both properties.
Many real-world non-convex problems, however, seem to possess both properties, either provably or empirically, and it is hard to explain why these nice properties, though seemingly different, often occur simultaneously.
Besides, most literature on local optimality focuses on problems without constraints or with tractable constraints.
This is usually the case for problems in the learning area.
However, for problems arising in cyber-physical systems, the constraints could include non-convex functions enforced by physical laws, as we will see in power systems.
In these cases, either the feasible set is not a Riemannian manifold, or the Riemannian gradient and Hessian are very hard to derive.
These questions motivate us to study conditions, sufficient or necessary, for problems to simultaneously have exact relaxation and no spurious local optima.
These conditions also help us study local optimality using properties of its relaxation, instead of its landscape.

Our conditions have two parts.
The first part, which also appeared in \cite{ZhoL20}, is on the sufficient condition.
Roughly, if for any relaxed point, there exists a path connecting it to the non-convex feasible set and the path satisfies  the following:
\begin{itemize}
\item along the path the cost is non-increasing,
\item along the path the `distance' to the non-convex feasible set is non-increasing,
\end{itemize}
then the problem must have exact relaxation and no spurious local optima simultaneously.
Here the `distance' can be any properly constructed function, as we will define later as a Lyapunov-like function (\cref{def:Lyapunov}).
The second part is on the necessary condition,
which says that if a problem does have exact relaxation and no spurious local optima simultaneously,
then there must exist such a Lyapunov-like function and paths satisfying the requirements above.
\footnote{The necessary condition is based upon some stronger assumptions so the second part is not the exact converse of the first part.}

Though Lyapunov-like functions and paths are guaranteed to exist, for specific problems it could still be difficult to construct them.
We then derive certain rules to construct a Lyapunov-like function and paths of a new problem from primitive problems with known Lyapunov-like functions and paths.
This process allows us to reuse and extend known results as the problem changes and grows.
Finally, we apply the proposed approach to two specific problems, optimal power flow (OPF) and low rank SDP.
Our work proves the first known condition (that can be checked {\it a priori}) for OPF to have no spurious local optima,
and it helps explain the widespread empirical experience that local algorithms for OPF problems often work extremely well.

\subsection{Background for Power Systems}
As one of the applications and main motivation of this work, OPF is a core problem in power systems.
First proposed in \cite{carpentier1962contribution}, OPF is a class of optimization problems that minimizes a certain cost subject to nonlinear physical laws and operational constraints.
It is known to be non-convex and NP-hard in its AC formulation \cite{Verma2009, lavaei2012zero, lehmann2016ac}.
Therefore, there is no known efficient algorithm that can solve all problem instances in polynomial time.
Traditional approaches to solving OPF are usually based on local algorithms such as Newton-Raphson, see \cite{momoh1999reviewa, momoh1999reviewb, jabr2002primal} for examples.
Over the past decade, techniques on convex relaxation have also been introduced to solve OPF \cite{jabr2006radial, Bai2008}.
A surprising empirical finding in the literature shows that despite the non-convexity, both local algorithms and convex relaxations very often yield global optimum of the original non-convex problem \cite{jabr2006radial, Bai2008, lavaei2012zero, bose2015quadratically}.
In recent years, there have been considerable analytical works on provable conditions for the relaxation exactness,
which are summarized in the reviews \cite{low2014convex, molzahn2019survey} and references therein.
However, few analytical results are known on the performance guarantee of local algorithms. 
In this paper, we show that a known sufficient condition for exact relaxation is also sufficient for local optima to be globally optimal.
To the best of our knowledge, this is the first analytical result of its kind, and we hope that the approaches developed in this paper can help derive more sufficient conditions along this direction.
\section{Preliminaries}
\label{sec:prelim}

In this paper, we will use $\K$ to denote the set $\Real$ of real numbers or the set $\Complex$ of complex numbers.
For any finite positive integer $n$,  $\K^n$ is a Banach space.

Consider a (potentially non-convex) optimization problem
\begin{subequations}
\begin{eqnarray}
\underset{x}{\text{minimize}}  && f(x)
\label{eq:opt.a}\\
\text{\quad subject to}& &x\in\X
\label{eq:opt.b}
\end{eqnarray}
\label{eq:opt}
\end{subequations}
and its convex relaxation
\begin{subequations}
\begin{eqnarray}
\underset{x}{\text{minimize}}  && f(x)
\label{eq:optR.a}\\
\text{\quad subject to}& &x\in\XR.
\label{eq:optR.b}
\end{eqnarray}
\label{eq:optR}
\end{subequations}
Here $\X$ is a nonempty compact subset of $\K^n$, not necessarily convex, while
$\XR\subseteq\K^n$ is an arbitrary compact and convex superset of $\X$.
The cost function $f:\XR\to\Real$ is convex and continuous over $\XR$.
We do not require the relaxation $\XR$ to be efficiently represented.

\begin{definition}
A point $\xlo\in\X$ is called a \emph{local optimum} of \eqref{eq:opt} if there exists a $\delta>0$ such that 
$f(\xlo)\leq f(x)$ for all $x\in\X$ with $\|x-\xlo\|<\delta$.  
\end{definition}
\begin{definition}[Strong Exactness]\label{def:exact}
We say the relaxation \eqref{eq:optR} is \emph{exact} with respect to \eqref{eq:opt} if every optimal point of \eqref{eq:optR} is feasible, 
and hence globally optimal, for \eqref{eq:opt}.
\end{definition}

Unless otherwise specified, we will always use the term {\it exact} to refer to such strong exactness.
\cref{def:exact} implies in particular that, if  \eqref{eq:optR} is {exact}, then $\forall\hat{x}\in \XR\setminus\X$, $f(\hat{x})> \min_{x\in \XR} f(x)$.  

\begin{definition}
A \emph{path} in $\set\subseteq\K^n$ connecting point $a$ to point $b$ is a continuous function $h:[0,1]\rightarrow\set$ such that $h(0)=a$ and $h(1)=b$. 
\end{definition}

We may refer to a path as the corresponding function $h$ in the remainder of the paper.

\begin{lemma}
\label{thm:A=B}
The following are equivalent:
\begin{enumerate}[label=(\Alph*)]
\item Problem \eqref{eq:optR} is exact with respect to \eqref{eq:opt}.
\item For any $x\in\XR\setminus\X$, there is a path $h$ in $\XR$ such that $h(0)=x, h(1)\in\X$, $f(h(t))$ is non-increasing for $t\in[0,1]$ and $f(h(0))>f(h(1))$. 
\end{enumerate}
\end{lemma}
\begin{proof}
(A)$\implies$ (B):
Let $x^*$ be any optimal point of \eqref{eq:optR}. By (A), $x^*\in\X$, thus for $x\in\XR\setminus\X$, we could choose the path as the line segment from $x$ to $x^*$ since $\XR$ is convex.

(B)$\implies$ (A): Condition (B) implies that no point $x\in\XR\setminus\X$ can be optimal for \eqref{eq:optR}.
\end{proof}

\cref{thm:A=B} is not surprising, and in fact many works in the literature proving exact relaxations of Optimal Power Flow problems can be interpreted as using (B) to prove (A) by implicitly finding such a path $h$ for each $x\in\XR\setminus\X$ \cite{low2014convex}.

Condition (B) does not say anything about the local optima in $\X$ for \eqref{eq:opt}. 
In the next section we will strengthen (B) by equipping the path with a Lyapunov-like function
and show that the stronger condition implies that all local optima of \eqref{eq:opt} are globally optimal.
We start by classifying local minima.

\begin{definition}
We classify each local optimum $\xlo$ of \eqref{eq:opt} into three disjoint classes: $\xlo$ is a
\begin{itemize}
\item {\it Global optimum (g.o.)} if $f(\xlo)\leq f(x)$ for all the feasible $x\in \X$.
\item {\it Pseudo local optimum (p.l.o.)} if there is a path $h:[0,1]\rightarrow \X$ such that $h(0)=\xlo$, $f(h(t))\equiv f(\xlo)$ for all $t\in[0,1]$ and $h(1)$ is not a local optimum. 
\item {\it Genuine local optimum (g.l.o.)} if it is neither a global optimum nor a pseudo local optimum.
\end{itemize}
\end{definition}

Examples of all three classes are shown in \cref{fig:lo}. 
\begin{figure}[htbp]
\centering
\includegraphics[width=\columnwidth]{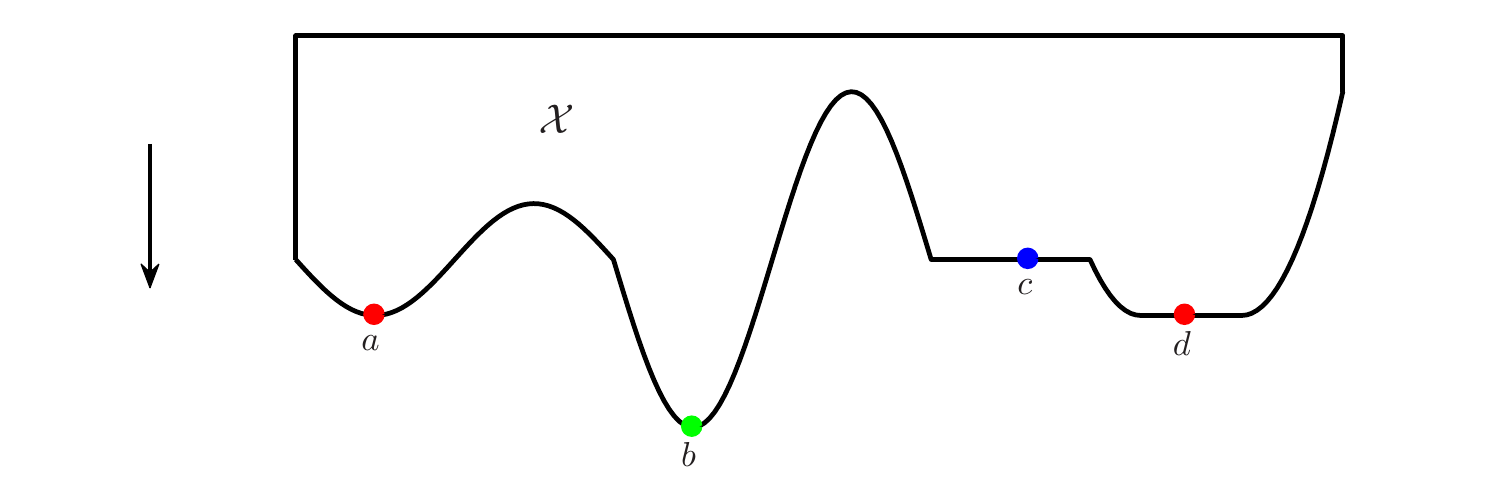}
\caption{Examples for three classes of local optima. The arrow indicates the direction along which the cost function decreases. Point $b$ is a global optimum, point $c$ is a pseudo local optimum, while points $a$ and $d$ are genuine local optima.}
\label{fig:lo}
\end{figure}

\begin{definition}
\label{def:improvable}
A point $x$ is {\it improvable} in $\X$ if there is a path $h:[0,1]\rightarrow \X$ such that
\begin{itemize}
\item $h(0)=x$;
\item $f(h(t))$ is non-increasing for $t\in[0,1]$;
\item $h(1)$ is not a local optimum or $f(h(1))<f(x)$.
\end{itemize}
\end{definition}
\begin{remark}
A local optimum is a pseudo local optimum if and only if it is improvable in $\X$.
\end{remark}
\begin{definition}
A set $\{h_i:i\in\mathcal{I}\}$ of paths indexed by $i$ is said to be \emph{uniformly bounded} if
there is a finite number $M$ such that $\|h_i(t)\|_\infty \leq M$ for every $i\in\mathcal{I}$ and $t\in[0,1]$.

\end{definition}
\begin{definition}
A set $\{h_i:i\in\mathcal{I}\}$ of paths indexed by $i$ is said to be \emph{uniformly equicontinuous} if 
for any $\epsilon>0$, there exists a $\delta>0$ such that
$\|h_i(t_1)-h_i(t_2)\|_\infty< \epsilon$ for every $i\in\mathcal{I}$ whenever $|t_1-t_2|<\delta$.
\end{definition}
\begin{remark}\label{rmk:linear}
The index set $\mathcal{I}$ could be empty or uncountably infinite. An empty path set (i.e., when $\mathcal{I}=\emptyset$) is considered to be both uniformly bounded and uniformly equicontinuous.
\end{remark}

Let $\Pi|_a^b$ be the family of all the finite ordered subsets of $[a,b]$. 
We use $\Pi$ as a shorthand for $\Pi|_0^1$. 
For $\pi=(t_0,\cdots,t_N)\in\Pi$ and a path $h$, define 
\begin{align*}
L_{\pi}(h):=\sum_{i=1}^N \|h(t_{i-1})-h(t_i)\|_{\ell_2}.
\end{align*}
Clearly, $L_{\pi}(h)$ is always finite for given $\pi$ and $h$.
\begin{definition}[\cite{toponogov2006differential}]
For path $h$, define the function
$L(h):=\sup_{\pi\in\Pi}L_\pi(h)$.
We say $h$ is rectifiable iff $L(h)$ is finite.
When $h$ is rectifiable, $L(h)$ is also referred to as its length.
\end{definition}
\begin{definition}[\cite{toponogov2006differential}]
For a rectifiable path $h:[0,1]\to\K^n$, let its {\it arc-length reparameterization} be $\bh:[0,1]\to\K^n$ and
\begin{align*}
\left\{
\begin{array}{ll}
\bh\Big(
\frac{1}{L(h)}\sup_{\pi\in\Pi|_0^t}L_\pi(h)
\Big):=h(t), & \text{if $L(h)>0$}\\
\bh:=h, & \text{if $L(h)=0$}
\end{array}
\right.
\end{align*}
\end{definition}
One could see $L(\bh)=L(h)<\infty$ and they have the same function image, i.e., $\{\bh(t)|t\in[0,1]\}=\{h(t)|t\in[0,1]\}$. 
For $0\leq t_1\leq t_2 \leq 1$, $\bh$ has the property that $\sup_{\pi\in\Pi|_{t_1}^{t_2}}L_\pi(\bh)=(t_2-t_1)L(\bh)$.
\begin{lemma}\label{lm:rect_equicon}
For a set of rectifiable paths $h_i, i\in\setI$, if the values of $L(h_{i})$ are uniformly bounded, then the set of $\bh_i, i\in\setI$ is uniformly equicontinuous.
\end{lemma}
\begin{proof}
Assume $L(h_{i})\leq M$ for all $i\in\setI$, then for any $0\leq t_1 \leq t_2\leq1$, we have for any $i$,
\begin{align*}
&\|\bh_i(t_1)-\bh_i(t_2)\|_{\infty}
\leq
\|\bh_i(t_1)-\bh_i(t_2)\|_{\ell_2}\\
\leq
&\sup_{\pi\in\Pi|_{t_1}^{t_2}}L_\pi(\bh_i)
=(t_2-t_1)L(h_i)
\leq M|t_1-t_2|.
\end{align*}
Setting $\delta=\epsilon/M$, the equicontinuity is proved.
\end{proof}

\begin{corollary}\label{cor:linear}
If $\set$ is compact in $\K^n$ and all paths in a set $\mathcal{H}=\{h_i:i\in\mathcal{I}\}$ are $[0,1]\to\set$ and consist of at most $N$ linear segments, then $\{\overline{h_i}:i\in\mathcal{I}\}$ must be both uniformly bounded and uniformly equicontinuous.
Here $N$ is a finite constant for all paths in $\mathcal{H}$.
\end{corollary}

\section{Sufficient Conditions}
\label{sec:suff}

In this section, we first study the sufficient conditions under which \eqref{eq:optR} is exact w.r.t. \eqref{eq:opt} and all the local optima of \eqref{eq:opt} are also globally optimal.
Those sufficient conditions will be proposed by strengthening Condition (B).
Note that (B) has already implied \eqref{eq:optR} is exact w.r.t. \eqref{eq:opt},
so our strategy is to strengthen (B) in order to rule out the possibility of genuine local optima and pseudo local optima.
\subsection{Ruling Out Genuine Local Optima}

\begin{definition}\label{def:Lyapunov}
A \emph{Lyapunov-like function}
\footnote{In contrast to a standard Lyapunov function, we do not require $V$ to be differentiable here.}
associated with \eqref{eq:opt} and \eqref{eq:optR} is a continuous function $V:\XR\rightarrow\Real^+$ such that
$V(x)=0$ for $x\in\X$ and $V(x)>0$ for $x\in\XR\setminus\X$.
\end{definition}

A strengthened version of (B) is as follows.
\begin{enumerate}[label=(\Alph*)]
\setcounter{enumi}{2}
\item There exists a Lyapunov-like function $V$ associated with \eqref{eq:opt} and \eqref{eq:optR} such that:
\begin{enumerate}[label=(C\arabic*)]
\item\label{C1} For any $x\in\XR\setminus\X$, there is a path $h_x$ in $\XR$ such that $h_x(0)=x, h_x(1)\in\X$, both $f(h_x(t))$ and $V(h_x(t))$ are non-increasing for $t\in[0,1]$ and $f(h_x(0))>f(h_x(1))$. 
\item\label{C2} The set $\{h_x\}_{x\in\XR\setminus\X}$ is uniformly bounded and uniformly equicontinuous.
\end{enumerate}
\end{enumerate}

\begin{theorem}\label{thm:suff}
If (C) holds, then (A) also holds and any local optimum in $\X$ for \eqref{eq:opt} is either a global optimum or a pseudo local optimum.
\end{theorem}

\begin{figure}
\centering
\vspace*{0.5em}
\includegraphics[width=0.9\columnwidth]{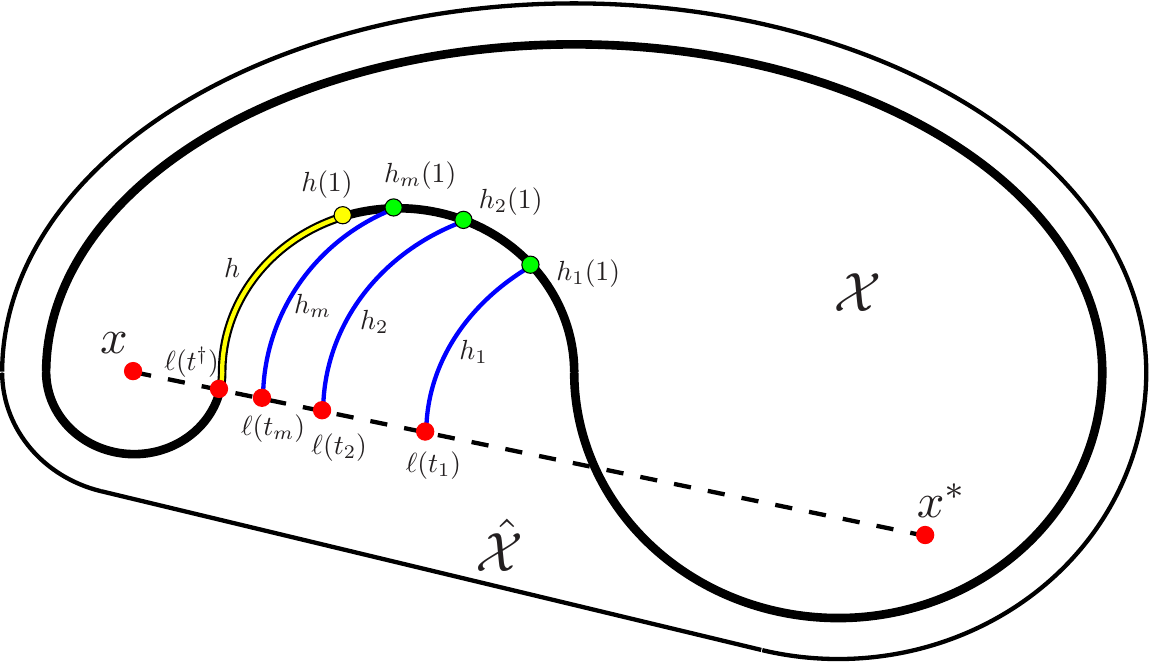}
\caption{Sketch of notations for the proof of \cref{thm:suff}. Point $x$ and $\ell(\tlim)$ will be later proved to be identical to each other.}
\label{fig:pf}
\end{figure}

\begin{proof}
(C) $\implies$ (A) is because (C) is stronger than (B).
As for the second part of the argument,
we include an illustrative sketch of the notations in \cref{fig:pf}.
Suppose $x\in\X$ is a local but not global optimum for \eqref{eq:opt}. We will prove that $x$ must be improvable in $\X$ (and thus a pseudo local optimum).

Let $x^*\neq x$ be a global optimum of \eqref{eq:opt}, so $f(x^*) < f(x)$.
Let $\ell:[0,1]\rightarrow\XR$ be the linear function characterizing the line segment from $x$ to $x^*$, i.e., $\ell(t) = (1-t)x + tx^*$
with $f(\ell(1))=f(x^*)<f(x)$.
Note that $f(\ell(t))$ is non-increasing in $t$.  To see this, consider any $t\geq 0,\epsilon>0$ with $t+\epsilon\leq 1$,
$x_1 = \ell(t)$, $x_2 = \ell(t+\epsilon)$.  Setting $s:=\epsilon/(1-t)$, we have $x_2 = (1-s)x_1 + s x^*$.
Since  $f$ is convex and $x^*$ is also a global optimum of \eqref{eq:optR} over $\XR$, 
we have
\begin{align*}
f(x_2) &\ \leq \ (1-s)f(x_1) + sf(x^*) \ \leq \ f(x_1).
\end{align*}
Define
\begin{align*}
\tlim:=&\sup_{t\in[0,1]} t\ \ \text{s. t.}~ \ell(\tau)\in\X\ \ \forall \tau\leq t.
\end{align*}
As $\X$ is closed, $\ell(\tlim)$ is also in $\X$. 
We first prove $\ell(\tlim)$ must be $x$ (i.e., $\tlim=0$).
Otherwise, as $x$ is a local optimum, we could find $\delta\in(0,\tlim)$ such that $f(\ell(t))\geq f(\ell(0))=f(x)$ for all $t\in[0,\delta)$. 
Since $f(\ell(t))$ is non-increasing in $t$, we must have $f(\ell(t))\equiv f(\ell(0))=f(x)$ for all $t\in[0, \delta)$.
It contradicts the fact that $f(\ell(t))$ is convex and $f(\ell(1))=f(x^*)<f(x)=f(\ell(0))$ for the same reason $f$ is non-increasing in $t$.

Therefore $\ell(\tlim)=x$ and $f(\ell(\tlim))=f(x)$.
It is sufficient to show $\ell(\tlim)$ is improvable in $\X$.
That is to say, it is sufficient to find some function $h:[0,1]\rightarrow\X$ such that $h(0)=\ell(\tlim)$, $f(h(t))$ is non-increasing in $t\in[0,1]$ and either $f(h(1))<f(\ell(\tlim))$ or $h(1)$ is not a local optimum in $\X$ for \eqref{eq:opt}. 

By the definition of $\tlim$, there is a decreasing sequence $t_m \rightarrow \tlim$ such that $t_m\in(\tlim,1]$ and $\ell(t_m)\in\XR\setminus\X$ for all $m$.
Since $f(\ell(t))$ is non-increasing in $t$, the sequence $f(\ell(t_m))$ is non-decreasing in $m$ and $f(\ell(t_m))< f(\ell(\tlim))$.
\footnote{The strict inequality is due to the convexity of $f(\ell(t))$ and the fact that $f(\ell(1))<f(\ell(\tlim))$.}
For each $\ell(t_m)$ we take the function $h_m:[0,1]\rightarrow\XR$ guaranteed by Condition (C).
As the sequence $h_m$ is uniformly bounded and uniformly equicontinuous, a subsequence must uniformly converge to a limit $h$ by Arzel\`a-Ascoli theorem.
Without loss of generality, we denote this subsequence as $h_m$ as well.
Next we prove this $h$ satisfies all the properties in \cref{def:improvable}, implying the improvability of $x$.

To show $h(t)\in\X$ for any fixed $t\in[0,1]$, we consider the sequence $(V(h_m(t)):m\in\Zahlen)$. 
As $\XR$ is closed, we have $h(t)=\lim_{m\to\infty}h_m(t)\in\XR$.
Further consider $V$ is continuous and $V(h_m(t))\leq V(h_m(0))$, thus
\begin{align*}
0\leq &V(h(t)) = V(\lim_{m\to\infty}h_m(t)) = \lim_{m\to\infty}V(h_m(t))\\
\leq & \lim_{m\to\infty}V(h_m(0))
=\lim_{m\to\infty}V(\ell(t_m))=V(\ell(\tlim))=0.
\end{align*}
Hence $V(h(t))=0$ and $h(t)\in\X$.

To show $h(0)=\ell(\tlim)$, we consider
\begin{align*}
h(0)=\lim_{m\to\infty}h_m(0)=\lim_{m\to\infty}\ell(t_m)=\ell(\tlim).
\end{align*}

To show $f(h(t))$ is non-increasing, we take any $s,t\in[0,1]$ such that $s<t$. As $f$ is continuous, we have
\begin{align*}
f(h(s))&=\lim_{m\to\infty}f(h_m(s))\\
f(h(t))&=\lim_{m\to\infty}f(h_m(t))
\end{align*} 
and by Condition (C) we have $f(h_m(s))\geq f(h_m(t))$ for each $m$. Therefore $f(h(s))\geq f(h(t))$.

Finally, we will show if $f(h(1))=f(\ell(\tlim))$ then $h(1)$ must not be a local minimal in $\X$ for \eqref{eq:opt}. 
For each $m$, 
\begin{align*}
f(h_m(1))\leq f(h_m(0))=f(\ell(t_m))< f(\ell(\tlim)) = f(h(1))
\end{align*}
and $h_m(1)\in\X$.
Since the sequence $h_m(1)$ converges to $h(1)$ as $m\to\infty$, within any open neighborhood of $h(1)$ in $\X$, we could always find some $h_m(1)$ with strictly smaller cost value.
Thus $h(1)$ cannot be a local minimum in $\X$.
\end{proof}

\subsection{Ruling Out Pseudo Local Optima}
So far, Condition (C) has eliminated the possibility of having genuine local optima, and in this subsection we further strengthen the condition to also rule out pseudo local optima.

Consider the following lemma and its corollaries.
\begin{lemma}\label{lm:connectedness}
If \eqref{eq:opt} is exact with respect to \eqref{eq:optR} and \eqref{eq:opt} has no genuine local optima, then the feasible set of \eqref{eq:opt} is connected.
\end{lemma}
\begin{proof}
If $\X$ is not connected, then by definition $\X$ can be partitioned into two disjoint non-empty closed sets $\X_1$ and $\X_2$
with $\X = \X_1 \cup \X_2$, which are hence both compact.
Further we let $x_i$ be any global optimum of $\min_{x\in\X_i}f(x)$ for $i=1,2$. 
Clearly $x_1\neq x_2$ and they are both local optima of \eqref{eq:opt}.

If $f(x_1)=f(x_2)$, then any convex combination of $x_1$, $x_2$ must be a global optimum to \eqref{eq:optR}. Since there is no path in $\X$ that connects $x_1$ and $x_2$, there must be some convex combination that is outside $\X$. This contradicts the exactness of relaxation.

If $f(x_1)\neq f(x_2)$, without loss of generality we assume $f(x_1)< f(x_2)$, i.e., $x_2$ is not a global optimum of \eqref{eq:opt}.
But $x_2$ is not a pseudo local optimum of \eqref{eq:opt} either, contradicting \cref{thm:suff}.
To see this, note that any point $x'\in\X$ which is connected to $x_2$ via a path in $\X$ must also be a point in $\X_2$ and if $f(x')=f(x_2)$ then $x'$ must be a local optimum of \eqref{eq:opt} as well. 
\end{proof}

\begin{corollary}
Condition (C) implies that the feasible set of \eqref{eq:opt} is connected.
\end{corollary}

Now we are in a good position to discuss some conditions that rule out pseudo local optima and therefore guarantee that any local optimum must be a global optimum.
\begin{corollary}
If all local optima of \eqref{eq:opt} are isolated, then Condition (C) implies that any local optimum of \eqref{eq:opt} is a global optimum.\end{corollary}

Here, local optima being isolated means any local optimum of \eqref{eq:opt} has an open neighborhood which contains no other local optimum.
The proof is straightforward as by definition isolated local optimum could not be pseudo local optimum.
In fact, in this case the optimum can be proved to be also unique.

Another way to eliminate pseudo local optima is by strengthening the monotonicity of $f(h_x(t))$ in Condition (C).
Consider the following condition which is slightly stronger than (C).

\begin{enumerate}[label=(\Alph*')]
\setcounter{enumi}{2}
\item Condition (C) holds, and there exists $k>0$ such that $\forall x\in\XR\setminus\X$, $\forall 0\leq t<s\leq 1$ we have
\begin{align}\label{eq:fh_lowerbound}
f(h_x(t))-f(h_x(s))\geq k \|h_x(t)-h_x(s)\|.
\end{align}
\end{enumerate}

In Condition (C'), $\|\cdot\|$ could be any norm on $\K^n$.
As a caveat, $\ell_0$-``norm'' is {\it not} allowed here as it is not a norm since it does not satisfy 
$\|\alpha x\| = |\alpha| \|x\|$.
Note that Condition (C) already implies $f(h_x(t))-f(h_x(s))\geq 0$,
while (C') strengthens this condition by enforcing a positive lower bound depending on $h_x$.

\begin{theorem}\label{co:lo_is_go}
If (C') holds, then any local optimum of \eqref{eq:opt} must be a global optimum.
\end{theorem}
\begin{proof}
Following the proof of \cref{thm:suff}, 
suppose $x\in\X$ is a local but not global optimum for \eqref{eq:opt}.
Then we have $x=\ell(\tlim)$ and could obtain a limit point of the sequence $h_m$, denoted as $h$.
Since both sides of \eqref{eq:fh_lowerbound} are continuous in $h_m(t)$ and $h_m(s)$, 
and the limits of $h_m(t)$ and $h_m(s)$ are $h(t)$ and $h(s)$,
we must have whenever $h(t)\neq h(s)$,
\begin{align*}
f(h(t))-f(h(s))\geq k \|h(t)-h(s)\|>0.
\end{align*}
Taking $t=0$ we can conclude that $h(0)$ (which is the same point as $x$) is not a local optimum of \eqref{eq:opt}.
\end{proof}
\section{Necessary Conditions}
In this section we will study the necessary conditions for a non-convex problem to have exact relaxation and no spurious local optima simultaneously.
It turns out the results are not exactly the converses of  \cref{thm:suff} or \cref{co:lo_is_go}, but in a slightly weaker sense.
Specifically, we show that if a non-convex problem is known to have exact relaxation and no spurious local optima simultaneously,
then the Lyapunov-like function and paths satisfying Condition (C) are guaranteed to exist.
However, it still may or may not be easy to find those functions or paths in practice for a specific problem.
\subsection{Results}
\begin{assumption}\label{as:semianalytic}
The feasible set $\X$ is semianalytic and the cost function $f$ is analytic.
\end{assumption}
We refer to \cite{bierstone1988semianalytic} for more detailed definitions and properties of semianalytic sets.
This assumption is not restrictive for most engineering problems.
If $\K$ is chosen as $\Complex$, then we suggest to view all the complex functions as functions of real variables by separating the real and imaginary parts,
and the space of $\Complex^n$ can be viewed as a shorthand for $\Real^{2n}$ in this section.

\begin{theorem}[necessary condition]\label{thm:nece}
If \eqref{eq:optR} is exact with respect to \eqref{eq:opt} and any local optimum of \eqref{eq:opt} is globally optimal, 
then there exists a Lyapunov-like function $V$ and a corresponding family of paths $\{h_x\}_{x\in\XR\setminus\X}$ satisfying \ref{C1} and \ref{C2}.
\end{theorem}
\begin{remark}
Note that \cref{thm:nece} is NOT the converse of \cref{thm:suff} in a strict sense. There are a few differences in their settings.
\begin{itemize}
\item  \cref{thm:suff} allows pseudo local optimum (in the conclusion) of the theorem, while  \cref{thm:nece} disallows it (in the premise).
\item  \cref{thm:nece} relies on  \cref{as:semianalytic} while  \cref{thm:suff} does not.
\end{itemize}
\end{remark}

\subsection{Proof Setup}
In the rest of the section, we will prove \cref{thm:nece}.
From now on, we assume \eqref{eq:optR} is exact with respect to \eqref{eq:opt} and any local optimum of \eqref{eq:opt} is also globally optimal.
We first have the following definition and lemmas, which are the main reasons we introduced \cref{as:semianalytic}.
\begin{definition}[Whitney regularity \cite{bierstone1988semianalytic, bierstone1980differentiable, hardt1983some}]
For a compact set $\setU\subset\K^n$ and a positive integer $p$, we say $\setU$ is {\it $p$-regular} if there exists $C>0$ such that $\forall x,y\in\setU$, $x$, $y$ can be joined by a rectifiable curve $h$ in $\setU$ satisfying $L(h)\leq C\|x-y\|^{1/p}$.
\end{definition}
\begin{lemma}[Theorem 6.10 in \cite{bierstone1988semianalytic}]\label{lm:whitney}
If $\setU$ is a compact connected subanalytic subset of $\K^n$, then there is a positive integer $p$ such that $\setU$ is $p$-regular and the curves can always be chosen semianalytic. 
\end{lemma}

The proof of \cref{lm:whitney} can be found in \cite{bierstone1988semianalytic}.
Note that any semianalytic set is also subanalytic.

\begin{lemma}\label{lm:short_curve}
For any $x_0\in\X$ that is not a local optimum of \eqref{eq:opt} and for any $\epsilon>0$,
there exists a path $h$ in $\X$ such that $h(0)=x_0$, $f(h(t))$ is non-increasing in $t$, $f(h(1))<f(h(0))$ and $L(h)<\epsilon$.
\end{lemma}
\begin{proof}
Consider the set $\setU:=\{x\in\X: f(x)\leq f(x_0)\}$, which by definition is also semi-analytic. 
Since $x_0\in\X$ is not an optimum of \eqref{eq:opt}, 
the problem $\min_{x\in\setU}f(x)$ must also be exact with respect to \eqref{eq:optR} and it does not introduce new local optima compared to \eqref{eq:opt}.
By \cref{lm:connectedness}, $\setU$ must be connected.

According to \cref{lm:whitney}, there is a rectifiable and semianalytic curve $h_0$ in $\setU$ such that $h_0(0)=x_0$, $L(h_0)<\epsilon$, $f(h_0(1))<f(x_0)$ and $f(h_0(t))\leq f(x_0)$ for all $t\in[0,1]$.
Here $f(h_0(1))$ can be chosen as any point in $\setU$ which has a strictly smaller cost value than $x_0$ and is sufficiently close to $x_0$ in Euclidean distance.
\footnote{We can always do so because $x_0$ is not a local optimum of \eqref{eq:opt}. The inequality $L(h_0)<\epsilon$ is satisfied because of the $p$-regularity of $\setU$.}
It is known that a semianalytic curve is analytic except for a finite number of points \cite{gabrielov1968projections}.
Assume $h_0(t)$ is not analytic at $0=a_0<a_1<\cdots<a_k=1$ where $k\geq 1$.
By Theorem on the parametrization of a semi-analytic arc in \cite{lojasiewicz1995semi} and the assumption that $f$ is analytic,
the value of $f(h_0(t))$ within any interval $[a_{\ell-1}, a_{\ell}]$ should be equal to some analytic function defined over an open superset of $[a_{\ell-1}, a_{\ell}]$.
Since $f(h_0(1))<f(h_0(0))$, the function $f(h_0(t))$ cannot be a constant function over $[0,1]$.
Let $[a_{\ell-1}, a_{\ell}]$ be the first interval within which $f(h_0(t))$ is not constant,
then $f(h_0(a_{\ell-1}))=f(h_0(0))$.
As $f(h_0(t))$ within $[a_{\ell-1}, a_{\ell}]$ equals to a analytic function defined over an open superset of $[a_{\ell-1}, a_{\ell}]$,
there must be a small subinterval $[a_{\ell-1}, a_{\ell-1}+\delta)$ for some $\delta>0$ within which we always have
\begin{align*}
f(h_0(t))=f(h_0(a_{\ell-1})) + \sum_{i=0}^\infty c_i (t-a_{\ell-1})^i,
\end{align*}
where the right hand side is the Taylor expansion of $f(h_0(t))$ at $a_{\ell-1}$.
Since $f(h_0(t))$ is not constant over $[a_{\ell-1}, a_{\ell}]$, the coefficients $c_i$ cannot all be zeros by the identity theorem.
Suppose $c_{i_0}$ is the first nonzero coefficient in the sequence $\{c_i\}_{i=0}^{\infty}$,
we have two cases.
If $c_{i_0}>0$, then $f(h_0(t))$ is strictly increasing within $[a_{\ell-1}, a_{\ell-1}+\delta')$ for some small positive $\delta'<\delta$.
It contradicts to the facts that $f(h_0(a_{\ell-1}))=f(x_0)$ and $f(h_0(t))\leq f(x_0)$ for all $t\in[0,1]$.
If $c_{i_0}<0$, then $f(h_0(t))$ is strictly decreasing within $[a_{\ell-1}, a_{\ell-1}+\delta')$ for some small positive $\delta'<\delta$.
Then we can construct a new path $h$ such that $h(t)=h_0(t\cdot(a_{\ell-1}+\delta'))$ for all $t\in[0,1]$.
It is easy to check such $h$ satisfies all the requirements in \cref{lm:short_curve}.
\end{proof}

Now we consider weaker versions of \ref{C1} and \ref{C2}.
\begin{enumerate}[label=(C\arabic*)]
\setcounter{enumi}{2}
\item\label{C3} For any $x\in\XR\setminus\X$, there is a path $h_x$ in $\XR$ such that $h_x(0)=x, h_x(1)\in\X$, both $f(h_x(t))$ and $V(h_x(t))$ are non-increasing for $t\in[0,1]$. 
\item\label{C4} All the $\{L(h_x)\}_{x\in\XR\setminus\X}$ are finite and uniformly bounded.
\end{enumerate}

Compared to \ref{C1}, \ref{C3} does not require $f(h_x(0))>f(h_x(1))$ to strictly hold. 
Then we have a weaker version of \cref{thm:nece} as follows.

\begin{lemma}[weaker necessary condition]\label{lm:weaker_nece}
If \eqref{eq:optR} is exact to \eqref{eq:opt} and any local optimum of \eqref{eq:opt} is also globally optimal, 
then there always exists a Lyapunov-like function $V$ and a corresponding family of paths $\{h_x\}_{x\in\XR\setminus\X}$ satisfying \ref{C3} and \ref{C4}.
\end{lemma}

We now show that \cref{lm:weaker_nece}, though weaker in its statement, actually implies \cref{thm:nece}, so later on we will only focus on the proof of \cref{lm:weaker_nece}.
To see this, we suppose $\Vd$ and $\{\hd_x\}_{x\in\XR\setminus\X}$ are the Lyapunov-like function and paths guaranteed by \cref{lm:weaker_nece}.

For each $x\in\XR\setminus\X$, if $\hd_x(1)$ is a local optimum (so it is also a global optimum) of \eqref{eq:opt} then we must have $f(h_x(1))<f(h_x(0))$ since the relaxation is exact. 
We construct $\hdd_x=\hd_x$.

If $\hd_x(1)$ is not a local optimum, then by \cref{lm:short_curve}, there exists a path $\hxtra_x$, which satisfies 
$\hxtra_x(0)=\hd_x(1)$, $\hxtra_x(t)\in\X$, $f(\hxtra_x(t))$ is non-increasing, $f(\hxtra_x(1))<f(\hxtra_x(0))$ and $L(\hxtra_x)<\epsilon$.
Here we choose $\epsilon$ as a fixed positive value for all $x$. We then construct
\begin{align*}
\hdd_x(t):=\left\{
\begin{array}{ll}
\hd_x(2t),&\text{if $t\in[0,1/2]$}\\
\hxtra_x(2t-1), &\text{if $t\in[1/2,1]$}.
\end{array}
\right.
\end{align*}

We now let $V=\Vd$ and $h_x=\overline{\hdd_x}$ for all $x\in\XR\setminus\X$. Recall that $\overline{\hdd_x}$ is the arc-length reparameterization of $\hdd_x$.
Clearly, such construction satisfies \ref{C1} as we strictly reduce the cost at the end of each path unless the path has already reached the global optimum. 
Besides, $\{h_x\}_{x\in\XR\setminus\X}$ is uniformly bounded as $\XR$ is compact, 
and $\{L(h_x)\}_{x\in\XR\setminus\X}$ also has the uniform upperbound as both $L(\hd_x)$ and $L(\hxtra_x)$ are uniformly bounded for all $x$.
By \cref{lm:rect_equicon} $\{h_x\}_{x\in\XR\setminus\X}$ is uniformly equicontinuous, so \ref{C2} is also satisfied.
To summarize, when \cref{lm:weaker_nece} is correct, one can always revise the Lyapunov-like function and paths provided by \cref{lm:weaker_nece} to make \cref{thm:nece} hold as well.
Therefore, in the rest of the section, we will only prove \cref{lm:weaker_nece} and the correctness of \cref{thm:nece} just follows.

Let $x^*$ be any global optimum of \eqref{eq:opt}, then it is also an optimum of \eqref{eq:optR}.
Define
\begin{align*}
\Hset:=&\{h~|~h: [0,1]\to\XR\text{~is~continuous and~} L(h)<\infty\}\\
\bHset:=&\{\bh~|~h\in\Hset\}\\
\tHset:=&\{h~|~h\in\bHset,~f(h(t))\geq f(h(1))\text{ for all $t\in[0,1]$}\}.
\end{align*}
An immediate observation is if a continuous function $h:[0,1]\to\XR$ satisfies $L(h)<\infty$ and $f(h(t))\geq f(h(1))$ for all $t$, then $\bh\in\tHset$.

\subsection{Construction}
We construct $V$ as
\begin{align}\label{eq:Vdef}
V(x)=\inf_{\substack{h\in\tHset\\h(0)=x\\h(1)\in\X}} L(h)
\end{align}
\begin{lemma}\label{lm:h_convergence}
For a sequence $(h_i)_{i=1}^\infty$ where $h_i\in\tHset$,
if both $(h_i)_{i=1}^\infty$ and $(L(h_i))_{i=1}^\infty$ are uniformly bounded,
then there must be a subsequence which uniformly converges to some $h^*$ such that its arc-length reparameterization, denoted as $\bhs$, is in $\tHset$.
Furthermore, $L(h^*)=L(\bhs)\leq \limsup_{i}L(h_i)$.
\end{lemma}
\begin{proof}
By \cref{lm:rect_equicon}, $(h_i)_{i=1}^\infty$ is both uniformly bounded and uniformly equicontinuous.
By Arzel\`a-Ascoli theorem, a subsequence of $(h_i)_{i=1}^\infty$ uniformly converges to a limit $h^*$.
Without loss of generalization, we denote this subsequence as $(h_i)_{i=1}^\infty$ as well.
By uniform limit theorem and the compactness of $\XR$, $\bhs$ is a continuous function mapping from $[0,1]$ to $\XR$.
To show $\bhs\in\tHset$, it is sufficient to show $f(h^*(t))\geq f(h^*(1))$ for all $t\in[0,1]$ and $L(h^*)<\infty$.
If $f(h^*(t))= f(h^*(1))-\epsilon$ for some $t\in[0,1]$ and $\epsilon>0$, then for sufficiently large $i$, we would have $|f(h_i(t))-f(h^*(t))|<\epsilon/3$ and $|f(h_i(1))-f(h^*(1))|<\epsilon/3$.
Thus $f(h_i(t)) \leq f(h_i(1))-\epsilon/3$ and it contradicts to $h_i\in\tHset$.

Instead of showing $L(h^*)<\infty$, we directly prove $L(h^*)\leq \limsup_{i}L(h_i)$.
Otherwise, there exists $\pi=(t_1,\cdots,t_N)\in\Pi$ such that $L_\pi(h^*)= \limsup_{i}L(h_i)+\epsilon$ for $\epsilon>0$.
For sufficiently large $i$, we have
\begin{align*}
&|L_\pi(h_i)-L_\pi(h^*)|\\
=&\Big|\sum_{j=1}^N \|h_i(t_{j-1})-h(t_j)\|_{\ell_2} - \sum_{j=1}^N \|h^*(t_{j-1})-h^*(t_j)\|_{\ell_2} \Big|\\
\leq&\sum_{j=1}^N\Big( \|h_i(t_{j-1})-h^*(t_{j-1})\|_{\ell_2} + \|h_i(t_{j})-h^*(t_{j})\|_{\ell_2}\Big)\leq\frac{\epsilon}{2}
\end{align*}
Thus, $L(h_i)\geq L_\pi(h_i) \geq \limsup_{i}L(h_i)+\epsilon/2$  holds for sufficiently large $i$. It contradicts to the definition of $\limsup$.
As a result, we must have $L(h^*)\leq \limsup_{i}L(h_i)$.
\end{proof}
\begin{lemma}
The optimization in \eqref{eq:Vdef} is feasible and the optimal cost can be achieved.
\end{lemma}
\begin{proof}
We fix some $x\in\XR$.
To show the feasibility, consider $\hfea(t):=(1-t)x+tx^*$, which is feasible to \eqref{eq:Vdef}.
Let $\Lfea=L(\hfea)$.
Since $\Lfea$ is finite and $L(h)$ is non-negative, $V(x)$ must be finite.
To show the achievability of the optimal cost, if not, then there must be a sequence of feasible $(h_i)_{i=1}^\infty$ such that 
\begin{align*}
&\Lfea > L(h_i) \geq L(h_{i+1}) > V(x)\text{ for all $i\geq1$}\\
&\lim_{i\to\infty}L(h_i)=V(x).
\end{align*}
The compactness of $\XR$ implies $(h_i)_{i=1}^\infty$ is uniformly bounded as well. 
By \cref{lm:h_convergence}, a subsequence of $(h_i)_{i=1}^\infty$, denoted as $(h_i)_{i=1}^\infty$ as well, uniformly converges to a limit $h^*$ and $L(h^*)=L(\bhs)\leq V(x)$.
Moreover, 
\begin{align*}
&\bhs(0)=h^*(0)=\lim_{i\to\infty}h_i(0)=x\\
&\bhs(1)=h^*(1)=\lim_{i\to\infty}h_i(1)\in\X.
\end{align*}

Above all, we proved $\bhs$ is feasible to \eqref{eq:Vdef}, and the cost $L(\bhs)$ is not worse than $V(x)$.
It contradicts to the non-achievability assumption.
\end{proof}

For each $x\in\XR\setminus\X$, we construct $h_x$ as
\begin{align}\label{eq:h}
h_x=\argmin_{\substack{h\in\tHset\\h(0)=x\\h(1)\in\X}} L(h)
\end{align}
If there are multiple minimizers then $h_x$ can be chosen as any one of them.
\begin{lemma}\label{lm:injective}
For $x\in\XR\setminus\X$, the function $h_x$ is injective.
\end{lemma}
\begin{proof}
Otherwise, for some $x$, there exist $t_1< t_2$ such that $h_{x}(t_1)=h_{x}(t_2)$.
Since $h_{x}\in\tHset\subseteq\bHset$, we have
\begin{align*}
\sup_{\pi\in\Pi|_{t_1}^{t_2}}L_\pi(h_{x})=(t_2-t_1)L(h_{x})=(t_2-t_1)V(x)>0
\end{align*}
Consider a new path defined as
\begin{align*}
h^*(t):=\left\{
\begin{array}{ll}
h_{x}(t),&\text{if $t\in[0,1]\setminus[t_1,t_2]$}\\
h_{x}(t_1), &\text{if $t\in[t_1,t_2]$}
\end{array}
\right.
\end{align*}
It is easy to check $h^*$ is continuous and entirely within $\XR$.
For any $t\in[0,1]$, $h_{x}\in\tHset$ implies $f(h^*(t))\geq f(h_{x}(1))=f(h^*(1))$.
Further, we have
\begin{align*}
L(h^*)=&\sup_{\pi\in\Pi}L_\pi(h^*)\\
=&\sup_{\pi\in\Pi|_{0}^{t_1}}L_\pi(h^*)+\sup_{\pi\in\Pi|_{t_1}^{t_2}}L_\pi(h^*)+\sup_{\pi\in\Pi|_{t_2}^{1}}L_\pi(h^*)\\
=&\sup_{\pi\in\Pi|_{0}^{t_1}}L_\pi(h_{x})+0+\sup_{\pi\in\Pi|_{t_2}^{1}}L_\pi(h_{x})\\
<&\sup_{\pi\in\Pi|_{0}^{t_1}}L_\pi(h_{x})+\sup_{\pi\in\Pi|_{t_1}^{t_2}}L_\pi(h_{x})+\sup_{\pi\in\Pi|_{t_2}^{1}}L_\pi(h_{x})\\
=&\sup_{\pi\in\Pi}L_\pi(h_{x})=L(h_{x}).
\end{align*}
Above all, the arc-length reparameterization of $h^*$, denoted as $\bhs$, is feasible to \eqref{eq:h} but achieves a strictly lower cost than $h_{x}$.
This contradicts to the optimality of $h_{x}$.
\end{proof}
\begin{corollary}\label{co:triangle}
For distinctive $t_1,t_2,t_3\in[0,1]$, if $f(h_x(t_1))\geq f(h_x(t_2))$ and $f(h_x(t_1))> f(h_x(t_3))$, then $\|h_x(t_1)-h_x(t_2)\|_{\ell_2} +\|h_x(t_1)-h_x(t_3)\|_{\ell_2}>\|h_x(t_2)-h_x(t_3)\|_{\ell_2}$.
\end{corollary}
\begin{proof}
It is sufficient to show $h_x(t_1)$ is not the convex combination of $h_x(t_2)$ and $h_x(t_3)$.
Otherwise, we assume  $h_x(t_1)=\lambda h_x(t_2) + (1-\lambda)h_x(t_3)$ for some $\lambda\in[0,1]$.
First, \cref{lm:injective} implies $\lambda\neq 0,1$.
For $\lambda\in(0,1)$, the convexity of $f$ implies
\begin{align*}
f(h_x(t_1))=&f(\lambda h_x(t_2) + (1-\lambda)h_x(t_3))\\
\leq &\lambda f(h_x(t_2)) + (1-\lambda)f(h_x(t_3))\\
<&\lambda f(h_x(t_1)) + (1-\lambda)f(h_x(t_1))=f(h_x(t_1)).
\end{align*}
This contradiction shows $h_x(t_1)$ is not the convex combination of $h_x(t_2)$ and $h_x(t_3)$. 
Then the triangle inequality implies this corollary.
\end{proof}
\begin{lemma}\label{lm:fht_non_incre}
For each $h_x$ defined in \eqref{eq:h}, we have $f(h_x(t))$ is non-increasing in $t$ for $t\in[0,1]$.
\end{lemma}
\begin{proof}
We fix an $x_0\in\XR\setminus\X$ and prove the result for $h_{x_0}$ defined above.
If not, then there exist $0\leq t_1<t_2\leq 1$ such that $f(h_{x_0}(t_1)) < f(h_{x_0}(t_2))$.
Now define
\begin{align*}
\td=\argmax_{t\in[t_1,1]}f(h_{x_0}(t))
,\quad
\tdd=\max_{\substack{t\in[\td,1]\\f(h_{x_0}(t))=f(h_{x_0}(\td))}}t.
\end{align*}
In other words, $\td$ is an arbitrary maximizer of $f(h_{x_0}(t))$, while $\tdd$ is the largest maximizer. Clearly, both $\td$ and $\tdd$ are well defined (due to the continuity and closedness) and are strictly between $t_1$ and $1$.
We also have $f(h_{x_0}(\tdd))>f(h_{x_0}(1))$ strictly holds.
By the continuity of $h_{x_0}(\cdot)$ and $f(h_{x_0}(\cdot))$, there exist $r,\delta>0$ such that $[\tdd-\delta, \tdd+\delta]\subseteq(t_1,1)$ and
\begin{itemize}
\item for $x\in\ball(h_{x_0}(\tdd),r)\cap\XR$, $f(h_{x_0}(1)) \leq f(x)$.
\item for $t\in[\tdd-\delta,\tdd)$, $h_{x_0}(t)\in\ball(h_{x_0}(\tdd),r)$. Therefore $f(h_{x_0}(1))\leq f(h_{x_0}(t))\leq f(h_{x_0}(\tdd))$.
\item for $t\in(\tdd,\tdd+\delta]$, $h_{x_0}(t)\in\ball(h_{x_0}(\tdd),r)$. Therefore $f(h_{x_0}(1))\leq f(h_{x_0}(t))< f(h_{x_0}(\tdd))$.
\end{itemize}

Now we construct another $h^*$ as
\begin{align*}
h^*(t)=\left\{
\begin{array}{l}
h_{x_0}(t),
 \hspace{4em} \text{if}~t\in[0,1]\setminus [\tdd-\delta,\tdd+\delta]\\
\frac{\tdd+\delta-t}{2\delta}h_{x_0}(\tdd-\delta) + \frac{t-\tdd+\delta}{2\delta}h_{x_0}(\tdd+\delta),\\
 \hspace{7.05em}  \text{if}~t\in [\tdd-\delta,\tdd+\delta].
\end{array}
\right.
\end{align*}
It is easy to verify $h^*$ is continuous.
For $t\in[\tdd-\delta,\tdd+\delta]$, $h^*(t)$ is the convex combination of $h_{x_0}(\tdd-\delta)$ and $h_{x_0}(\tdd+\delta)$, and must be within $\ball(h_{x_0}(\tdd),r)\cap\XR$, which is convex.
Therefore, $h^*$ is entirely within $\XR$ and $f(h^*(t))\geq f(h_{x_0}(1))=f(h^*(1))$ holds for all $t$.

Next, we are showing $L(h^*)<L(h_{x_0})$ by \eqref{eq:Lhstar}.
\begin{figure*}
\begin{align}
\nonumber
L(h^*)=&\sup_{\pi\in\Pi}L_\pi(h^*)
=\sup_{\pi\in\Pi|_{0}^{\tdd-\delta}}L_\pi(h^*)+\sup_{\pi\in\Pi|_{\tdd-\delta}^{\tdd+\delta}}L_\pi(h^*)+\sup_{\pi\in\Pi|_{\tdd+\delta}^{1}}L_\pi(h^*)\\
\nonumber
=&\sup_{\pi\in\Pi|_{0}^{\tdd-\delta}}L_\pi(h_{x_0})+  \|h_{x_0}(\tdd-\delta)-h_{x_0}(\tdd+\delta)\|_{\ell_2} +\sup_{\pi\in\Pi|_{\tdd+\delta}^{1}}L_\pi(h_{x_0})\\
\nonumber
<&\sup_{\pi\in\Pi|_{0}^{\tdd-\delta}}L_\pi(h_{x_0})+  \|h_{x_0}(\tdd-\delta)-h_{x_0}(\tdd)\|_{\ell_2} + \|h_{x_0}(\tdd)-h_{x_0}(\tdd+\delta)\|_{\ell_2} +\sup_{\pi\in\Pi|_{\tdd+\delta}^{1}}L_\pi(h_{x_0})\\
\leq&\sup_{\pi\in\Pi|_{0}^{\tdd-\delta}}L_\pi(h_{x_0})+ \sup_{\pi\in\Pi|_{\tdd-\delta}^{\tdd+\delta}}L_\pi(h_{x_0}) +\sup_{\pi\in\Pi|_{\tdd+\delta}^{1}}L_\pi(h_{x_0})
=\sup_{\pi\in\Pi}L_\pi(h_{x_0})=L(h_{x_0}).
\label{eq:Lhstar}
\end{align}
\hrule
\end{figure*}
The strict inequality in \eqref{eq:Lhstar} is because of \cref{co:triangle}.

Above all, the arc-length reparameterization of $h^*$, denoted as $\bhs$, is feasible to \eqref{eq:h} but achieves a strictly lower cost than $h_{x_0}$.
This contradicts to the optimality of $h_{x_0}$.
\end{proof}
\subsection{Verification}
\subsubsection{To show $V$ satisfies \cref{def:Lyapunov}}
It is sufficient to show $V$ is continuous in $x$. The proof is twofold.
To abuse the notations a little bit, we let $h_x(t) \equiv x$ for $x\in\X$, so such $h_x$ is the unique minimizer of \eqref{eq:Vdef} and $L(h_x)=V(x)=0$ for $x\in\X$. 

First we show for $x_0\in\XR$ and $\epsilon>0$, there exists $\deltap>0$ such that $\forall x\in \ball(x_0,\deltap)\cap\XR$, $V(x)\leq V(x_0)+\epsilon$.
There are two scenarios.
If $h_{x_0}(1)$ is a global optimum of \eqref{eq:opt},
then we could set $\deltap=\epsilon$. 
For any $x\in \ball(x_0,\deltap)\cap\XR$, construct
\begin{align*}
h^*(t)=\left\{
\begin{array}{ll}
(1-2t)x+ 2tx_0, & t\in[0,\frac{1}{2}]\\
h_{x_0}(2t-1), & t\in (\frac{1}{2},1].
\end{array}
\right.
\end{align*}
Its arc-length reparameterization $\bhs$ is feasible to \eqref{eq:Vdef} (w.r.t. $x$) and $V(x)\leq L(\bhs)=|x-x_0|+L(h_{x_0})\leq V(x_0)+\epsilon$.
Next we focus on the scenario that $h_{x_0}(1)$ is not a global optimum of \eqref{eq:opt}, so it is not a local optimum neither.
By \cref{lm:short_curve}, there is a path $\hxtra$ in $\X$ such that $\hxtra(0)=h_{x_0}(1)$, $f(\hxtra(t))$ is non-increasing in $t$, $f(\hxtra(1))<f(\hxtra(0))$ and $L(\hxtra)<\epsilon/2$.
Suppose $f(\hxtra(0))-f(\hxtra(1))=\tau>0$.
Since $f$ is continuous, there must be some $\gamma>0$ such that for any $x\in\ball(x_0,\gamma)\cap\XR$,  we have $|f(x)-f(x_0)|<\tau$.
Now we choose $\deltap$ as $\min(\gamma,\epsilon/2)$.
For any $x\in\ball(x_0,\deltap)\cap\XR$, construct
\begin{align*}
h^*(t)=\left\{
\begin{array}{ll}
(1-3t)x+ 3tx_0, & t\in[0,\frac{1}{3}]\\
h_{x_0}(3t-1), & t\in (\frac{1}{3},\frac{2}{3}]\\
\hxtra(3t-2), & t\in (\frac{2}{3},1].
\end{array}
\right.
\end{align*}
Its arc-length reparameterization $\bhs$ is feasible to \eqref{eq:Vdef} (w.r.t. $x$) and $V(x)\leq L(\bhs)=|x-x_0|+L(h_{x_0})+L(\hxtra)\leq \deltap + V(x_0)+\epsilon/2\leq V(x_0)+\epsilon$.

Second we show for $x_0\in\XR$ and $\epsilon>0$, there exists $\deltam>0$ such that $\forall x\in \ball(x_0,\deltam)\cap\XR$, $V(x)\geq V(x_0)-\epsilon$.
If not, then there must be a sequence $(x_i)_{i=1}^\infty$ such that $\lim_{i\to\infty}x_i=x_0$ but $V(x_i)<V(x_0)-\epsilon$ for all $i\geq 1$.
Let $h_i:=h_{x_i}$ for $i\geq 0$, then both $(h_i)_{i=1}^\infty$ and $(L(h_i))_{i=1}^\infty$ are uniformly bounded.
By \cref{lm:h_convergence}, a subsequence of $(h_i)_{i=1}^\infty$ uniformly converges to a limit $h^*$ and 
\begin{align*}
L(h^*)=&L(\bhs)\leq\limsup_{i}L(h_i)\\
=&\limsup_{i}V(x_i) \leq V(x_0)-\epsilon.
\end{align*}
\cref{lm:h_convergence} also indicates $\bhs\in\tHset$ and
$h^*(0)=\lim_{i\to\infty}h_i(0)=\lim_{i\to\infty}x_i=x_0$, $h^*(1)=\lim_{i\to\infty}h_i(1)\in\X$ (as $\X$ is closed).
Therefore, $\bhs$ is feasible to \eqref{eq:Vdef} but its cost is strictly lower than $V(x_0)$.
It leads to the contradiction.
\subsubsection{To show \ref{C3} holds}
By our construction \eqref{eq:h}, $h_x$ is entirely within $\XR$, and $h_x(0)=x,~h_x(1)\in\X$.
\cref{lm:fht_non_incre} shows $f(h_x(t))$ is non-increasing for $t\in[0,1]$.
It is sufficient to show $V(h_x(t))$ is also non-increasing for $t\in[0,1]$.
Consider the following lemma.
\begin{lemma}
For fixed $x_0\in\XR$ and $t_0\in[0,1]$, 
\begin{align*}
V(h_{x_0}(t_0))=\sup_{\pi\in\Pi|_{t_0}^{1}}L_\pi(h_{x_0})
\end{align*}
\end{lemma}
\begin{proof}
Let $x_1=h_{x_0}(t_0)$. We have $L(h_{x_1})=V(h_{x_0}(t_0))$. If the lemma does not hold, then we have two cases.

First, if $L(h_{x_1})<\sup_{\pi\in\Pi|_{t_0}^{1}}L_\pi(h_{x_0})$, then let
\begin{align*}
h^*(t)=\left\{
\begin{array}{ll}
h_{x_0}(2t_0t), & t\in[0,\frac{1}{2}]\\
h_{x_1}(2t-1), & t\in (\frac{1}{2},1].
\end{array}
\right.
\end{align*}
It is easy to check $h^*$ is continuous and entirely within $\XR$, and $h^*(0)=x_0, h^*(1)=h_{x_1}(1)\in\X$. For $t\in[0,1/2]$, 
\begin{align*}
f(h^*(t)) = &f(h_{x_0}(2t_0t)) \geq f(h_{x_0}(t_0)) = f(x_1)\\
 = &f(h_{x_1}(0))
 \geq f(h_{x_1}(1)) = f(h^*(1)).
\end{align*}
For $t\in[1/2,1]$,
\begin{align*}
f(h^*(t)) = f(h_{x_1}(2t-1)) \geq f(h_{x_1}(1)) = f(h^*(1)).
\end{align*}
Further, we have
\begin{align*}
L(h^*)=&\sup_{\pi\in\Pi}L_\pi(h^*)\\
=&\sup_{\pi\in\Pi|_{0}^{0.5}}L_\pi(h^*)+\sup_{\pi\in\Pi|_{0.5}^{1}}L_\pi(h^*)\\
=&\sup_{\pi\in\Pi|_{0}^{t_0}}L_\pi(h_{x_0})+\sup_{\pi\in\Pi}L_\pi(h_{x_1})\\
=&\sup_{\pi\in\Pi|_{0}^{t_0}}L_\pi(h_{x_0})+L(h_{x_1})\\
<&\sup_{\pi\in\Pi|_{0}^{t_0}}L_\pi(h_{x_0})+\sup_{\pi\in\Pi|_{t_0}^{1}}L_\pi(h_{x_0})=L(h_{x_0}).
\end{align*}
Above all, the arc-length reparameterization of $h^*$, denoted as $\bhs$, is feasible to \eqref{eq:h} (w.r.t. $x_0$) but achieves a strictly lower cost than $h_{x_0}$.
This contradicts to the optimality of $h_{x_0}$.

Second, if $L(h_{x_1})>\sup_{\pi\in\Pi|_{t_0}^{1}}L_\pi(h_{x_0})$, then let
\begin{align*}
h^*(t)=\left\{
\begin{array}{ll}
h_{x_0}(t_0), & t\in[0,t_0]\\
h_{x_0}(t), & t\in (t_0,1].
\end{array}
\right.
\end{align*}
It is easy to check $h^*$ is continuous and entirely within $\XR$, and $h^*(0)=h_{x_0}(t_0)=x_1, h^*(1)=h_{x_0}(1)\in\X$. 
For $t\in[0,1]$, $f(h_{x_0}(t))\geq f(h_{x_0}(1))$ implies $f(h^*(t))\geq f(h_{x_0}(1))=f(h^*(1))$.
Further, we have
\begin{align*}
&L(h^*)=\sup_{\pi\in\Pi}L_\pi(h^*)=\sup_{\pi\in\Pi|_{0}^{t_0}}L_\pi(h^*)+\sup_{\pi\in\Pi|_{t_0}^{1}}L_\pi(h^*)\\
=&0+\sup_{\pi\in\Pi|_{t_0}^{1}}L_\pi(h_{x_0})<L(h_{x_1}).
\end{align*}
Above all, the arc-length reparameterization of $h^*$, denoted as $\bhs$, is feasible to \eqref{eq:h} (w.r.t. $x_1$) but achieves a strictly lower cost than $h_{x_1}$.
This contradicts to the optimality of $h_{x_1}$.
\end{proof}

Using this lemma, we are in a good position to show $V(h_x(t))$ is non-increasing for $t\in[0,1]$.
For any $t_1<t_2$, we have
\begin{align*}
V(h_x(t_1)) = &\sup_{\pi\in\Pi|_{t_1}^{1}}L_\pi(h_{x})
= \sup_{\pi\in\Pi|_{t_1}^{t_2}}L_\pi(h_{x}) + \sup_{\pi\in\Pi|_{t_2}^{1}}L_\pi(h_{x})\\
\geq &\sup_{\pi\in\Pi|_{t_2}^{1}}L_\pi(h_{x}) =V(h_x(t_2)).
\end{align*}

\subsubsection{To show \ref{C4} holds}
The set $\{L(h_x)\}_{x\in\XR\setminus\X}$ is uniformly bounded by $\max_{x\in\XR}|x-x^*|$, which is finite.

To summarize, we have verified that such construction is well defined and satisfies both \ref{C3} and \ref{C4}, so \cref{lm:weaker_nece} is proved.
Since we have shown that \cref{lm:weaker_nece} implies \cref{thm:nece}, the latter is also proved.

\section{Other Properties}

\subsection{Constructing from Primitives}

\begin{table*}
\caption{A summary on constructing $V$ and $h_x$ from primitives.}
\begin{center}
\begin{tabular}{|m{25mm}|m{26mm}|m{35mm}|m{35mm}|m{30mm}|}
\hline
Operation & Primitive problem & New problem & $V,h_x$ for new problem & Additional requirements \\
\hline
Function Composition & $(f,\X,\XR)$: $V,h_x$ & $(g\circ f,\X,\XR)$ & $\tV:=V\newline \tpath_x:=h_x$ & $g$ is non-decreasing and convex\\
\hline
Union of feasible sets (with cost as the sum) & 
$(f_1,\X_1,\XR_1)$: $V^1,h_x^1$\newline $(f_2,\X_2,\XR_2)$: $V^2,h_x^2$ &
$(f,\X,\XR)$\newline where $f:=\lambda f_1+(1-\lambda)f_2$\newline $\X:=(\X_1\cup\X_2)\cap\XR_1\cap\XR_2$\newline $\XR:=\XR_1\cap\XR_2$&
$\tV:=V^1\times V^2\newline \tpath_x:=h_x^1$ &
$h_x^1$ and $h_x^2$ coincide for $x\in\XR\setminus\X$ \\
\hline
Union of feasible sets (with cost as the maximum)& 
$(f_1,\X_1,\XR_1)$: $V^1,h_x^1$\newline $(f_2,\X_2,\XR_2)$: $V^2,h_x^2$ &
$(f,\X,\XR)$\newline where $f:=\max(f_1,f_2)$\newline $\X:=(\X_1\cup\X_2)\cap\XR_1\cap\XR_2$\newline $\XR:=\XR_1\cap\XR_2$&
$\tV:=V^1\times V^2\newline \tpath_x:=h_x^1$ &
$h_x^1$ and $h_x^2$ coincide for $x\in\XR\setminus\X$ \\
\hline
Intersection of feasible sets (with cost as the sum)& 
$(f_1,\X_1,\XR)$: $V^1,h_x^1$\newline $(f_2,\X_2,\XR)$: $V^2,h_x^2$ \newline $(u,v)=:x\in\XR$&
$(f,\X,\XR)$\newline where $f:=\lambda f_1+(1-\lambda)f_2$\newline $\X:=\X_1\cap\X_2$&
$\tV:=V^1+ V^2\newline \tpath_x$ is defined as in \eqref{eq:path_for_ifs} &
$f_i(x), V^i(x)$ depend on $\proj_i x$ only and $\proj_{1-i} h_x^i$ is constant\\
\hline
Intersection of feasible sets (with cost as the maximum)& 
$(f_1,\X_1,\XR)$: $V^1,h_x^1$\newline $(f_2,\X_2,\XR)$: $V^2,h_x^2$ \newline $(u,v)=:x\in\XR$&
$(f,\X,\XR)$\newline where $f:=\max(f_1,f_2)$\newline $\X:=\X_1\cap\X_2$&
$\tV:=V^1+ V^2\newline \tpath_x$ is defined as in \eqref{eq:path_for_ifs} &
$f_i(x), V^i(x)$ depend on $\proj_i x$ only and $\proj_{1-i} h_x^i$ is constant\\
\hline
\end{tabular}
\end{center}
\label{tb:summary}
\end{table*}

Though the previous section guarantees the existence of the Lyapunov-like function and paths under certain conditions,
it is not clear how to systematically find or construct them. 
In this subsection, we show that if one can find the Lyapunov-like function and paths for some primitive problems, 
then there are natural ways to construct the Lyapunov-like function and paths for new problems built up from those primitives in certain ways.
To streamline the notations, we will use the tuple $(f,\X)$ to refer to \eqref{eq:opt} and the tuple $(f,\X,\XR)$ to refer to the problem pair \eqref{eq:opt}, \eqref{eq:optR}.
Assume $(V,\{h_x\}_{x\in\XR\setminus\X})$ is a valid construction of Lyapunov-like function and paths for $(f,\X,\XR)$.
In this subsection, when we say $V$ and $h_x$ are valid, it means they not only are valid by definition, but also satisfy \ref{C1} and \ref{C2}.

\subsubsection{Function Composition}
Suppose $g:\Real\to\Real$ is non-decreasing and convex.
Then $(V,\{h_x\}_{x\in\XR\setminus\X})$ is also a valid construction of Lyapunov-like function and paths for $(g\circ f,\X,\XR)$.
This result is trivial as $g\circ f$ preserves the convexity over $\XR$ and monotonicity over any path.


\subsubsection{Union of Feasible Sets}
Suppose for two pairs of problems $(f_1,\X_1,\XR_1)$, for which $(V^1,\{h_x^1\}_{x\in\XR_1\setminus\X_1})$ is valid, 
and $(f_2,\X_2,\XR_2)$, for which $(V^2,\{h_x^2\}_{x\in\XR_2\setminus\X_2})$ is valid.
We consider a new problem $(f,\X,\XR)$
where $\X:=(\X_1\cup\X_2)\cap\XR_1\cap\XR_2$ and $\XR:=\XR_1\cap\XR_2$.
The formulation of $f$ will be provided later.
If for any $x\in\XR\setminus\X$, we have $h_x^1\equiv h_x^2$,
then construct $\tV:\XR\to\Real$ such that $\tV(x):=V^1(x)\cdot V^2(x)$
and $\tpath_x=h_x^1$ for all $x\in\XR\setminus\X$.
We have the following two results.
\begin{corollary}\label{thm:U1}
For any $\lambda\in(0,1)$, define $f:\XR\to\Real$ as $f(x):=\lambda f_1(x)+ (1-\lambda) f_2(x)$.
Then $(\tV,\{\tpath_{x}\}_{x\in\XR\setminus\X})$ is valid for $(f,\X,\XR)$.
\end{corollary}
\begin{corollary}\label{thm:U2}
Define function $f:\XR\to\Real$ as $f(x):=\max( f_1(x), f_2(x))$.
Then $(\tV,\{\tpath_{x}\}_{x\in\XR\setminus\X})$ is valid for $(f,\X,\XR)$.
\end{corollary}

\begin{proof}[Proof for \cref{thm:U1} and \cref{thm:U2}]
The function $\tV$ is still continuous and vanishes if and only if $x\in\X$ (since $V(x)=0$ $\Leftrightarrow$ $V^1(x)=0$ {\it or} $V^2(x)=0$).
By construction, $\{h_x\}_{x\in\XR\setminus\X}$ is a subset of $\{h^1_x\}_{x\in\XR_1\setminus\X_1}$ so \ref{C2} is naturally satisfied.
To see \ref{C1} holds, we fix any $x\in\XR\setminus\X$. 
Then $h_x(0)=h_x^1(x)=x$ and $h_x(1)=h_x^1(1)\in\X_1\subseteq\X$.
Further, $V(h_x(t))=V^1(h_x(t))V^2(h_x(t))=V^1(h_x^1(t))V^2(h_x^2(t))$ as $h_x^1$ and $h_x^2$ conincide when $x\in\XR\setminus\X$.
Because both $V^1(h_x^1(t))$ and $V^2(h_x^2(t))$ are non-negative and non-increasing, so is $V(h_x(t))$.
Finally, as $f_1(h_x(t))$ and $f_2(h_x(t))$ are both non-increasing over $[0,1]$, their convex-combination or maximum (i.e., $f(h_x(t))$) must be non-increasing as well.
A similar argument can also be applied to show $f(h_x(1))<f(h_x(0))$. Thus \ref{C1} holds and it completes the proof.
\end{proof}

\subsubsection{Intersection of Feasible Sets}
We still consider two pairs of problems $(f_1,\X_1,\XR)$, for which $(V^1,\{h_x^1\}_{x\in\XR_1\setminus\X_1})$ is valid, 
and $(f_2,\X_2,\XR)$, for which $(V^2,\{h_x^2\}_{x\in\XR_2\setminus\X_2})$ is valid.
Different from the previous setting, two pairs are required to share the same relaxation set $\XR$.
Further, we view each $x\in\XR$ as a tuple with two parts $x:=(u,v)$.
Define $\proj_1$ and $\proj_2$ as two projection operators such that $\proj_1 x = u$ and $\proj_2 x=v$.

We consider a new problem $(f,\X,\XR)$
where $\X:=\X_1\cap\X_2$.
The formulation of $f$ will be provided later.

If $f_i, V^i$ and $h_x^i$ are completely separated with respect to $u$ and $v$ in the sense that for $i=1,2$,
$f_i(x), V^i(x)$ depend on $\proj_i x$ only and $\proj_{1-i} (h_x^i(t))$ is constant,
then we can construct $\tV$ as $\tV(x):=V^1(x)+ V^2(x)$.
For $x\in\XR\setminus\X$, the path $\tpath_x$ is constructed in three ways depending on the values of $V^1(x)$ and $V^2(x)$.
\begin{subequations}
\begin{eqnarray}
\label{eq:path_for_ifs.a}
&&\text{If}~V^1(x)=0~\text{then}~\tpath_x:=h_x^2,\\
\label{eq:path_for_ifs.b}
&&\text{If}~V^2(x)=0~\text{then}~\tpath_x:=h_x^1,\\
\nonumber
&&\text{If}~V^1(x),V^2(x)>0~\text{then}\\
\label{eq:path_for_ifs.c}
&&\hspace{5em} \tpath_x(t):=\left\{
\begin{array}{ll}
h_x^1(2t), & t\in[0,\frac{1}{2})\\
h_{h_x^1(1)}^2(2t-1), & t\in[\frac{1}{2},1]
\end{array}
\right..
\hspace{2em} 
\end{eqnarray}
\label{eq:path_for_ifs}
\end{subequations}
\begin{corollary}\label{thm:I1}
For any $\lambda\in(0,1)$, define $f:\XR\to\Real$ as $f(x):=\lambda f_1(x)+ (1-\lambda) f_2(x)$.
Then $(\tV,\{\tpath_{x}\}_{x\in\XR\setminus\X})$ is valid for $(f,\X,\XR)$.
\end{corollary}
\begin{corollary}\label{thm:I2}
Define function $f:\XR\to\Real$ as $f(x):=\max( f_1(x), f_2(x))$.
Then $(\tV,\{\tpath_{x}\}_{x\in\XR\setminus\X})$ is valid for $(f,\X,\XR)$.
\end{corollary}
\begin{proof}[Proof for \cref{thm:I1} and \cref{thm:I2}]
The function $\tV$ is still continuous and vanishes if and only if $x\in\X$ (since $V(x)=0$ $\Leftrightarrow$ $V^1(x)=0$ {\it and} $V^2(x)=0$).
The set $\{\tpath_{x}\}_{x\in\XR\setminus\X}$ satisfies \ref{C2} as each path is constructed either as $h_x^i$ or the concatenation of $h_x^1$ and $h_{h_x^1(1)}^2$.
Next we are showing $\tpath_x(1)$ is in $\X$.
If $\tpath_x$ is constructed by \eqref{eq:path_for_ifs.a},
then we have $\tV(\tpath_x(1))=V^1(h_x^2(1))+V^2(h_x^2(1))=V^1(h_x^2(1))$.
Since $V^1(h_x^2(1))$ only depends on $\proj_1 h_x^2(1)$ and $\proj_1 h_x^2(1)=\proj_1 h_x^2(0)=\proj_1 x$,
there must be $V^1(h_x^2(1))=V^1(x)=0$.
Thus $\tV(\tpath_x(1))=0$ and $\tpath_x(1)\in\X$.
It is similar if $\tpath_x$ is constructed by \eqref{eq:path_for_ifs.b}.
When $\tpath_x$ is constructed by \eqref{eq:path_for_ifs.c},
then 
\begin{align*}
\tV(\tpath_x(1))=&V^1(h_{h_x^1(1)}^2(1))+V^2(h_{h_x^1(1)}^2(1))\\
=&V^1(h_{h_x^1(1)}^2(1))=V^1(h_{h_x^1(1)}^2(0))\\
=&V^1(h_x^1(1))=0,
\end{align*}
so $\tpath_x(1)\in\X$ as well.
The monotonicity properties of $\tV(\tpath_x(t))$ and $f(\tpath_x(t))$ are also the direct consequence of the fact that $f_i, V^i$ and $h_x^i$ are completely separated.
\end{proof}

A summary of this subsection has been provided in \cref{tb:summary}.

\subsection{Weak Exactness}
One observation from the proof of \cref{thm:suff} is we do not actually need $f(h_x(0))>f(h_x(1))$ to eliminate genuine local optima.
However, such strict inequality is required to show the exactness.
We can consider a weaker version of exactness defined as follows.
\begin{definition}[Weak Exactness]\label{def:weak_exact}
We say the relaxation \eqref{eq:optR} is \emph{weakly exact} with respect to \eqref{eq:opt} if at least one optimum of \eqref{eq:optR} is feasible, 
and hence globally optimal, for \eqref{eq:opt}.
\end{definition}
\begin{theorem}\label{thm:weak_suff}
If there exists a Lyapunov-like function $V$ associated with \eqref{eq:opt} and \eqref{eq:optR} such that \ref{C3} and \ref{C2} hold,
then \eqref{eq:optR} is weakly exact with respect to \eqref{eq:opt}
and any local optimum in $\X$ for \eqref{eq:opt} is either a global optimum or a pseudo local optimum.
\end{theorem}
The argument on weak exactness follows from the fact that the path connects any global optimum of \eqref{eq:optR} must determine a endpoint in $\X$ with the same cost, 
which by definition must be a global optimum as well.
The argument on local optimality follows directly from the proof of \cref{thm:suff}.
\section{Applications}\label{sec:application}
In this sections we will use two examples to show for specific problems, what $V$ and $\{h_x\}$ might look like.
The first example is Optimal Power Flow (OPF) problem in power systems with tree structres, which is also the motivating problem for us to develop this theory.
By finding the Lyapunov-like function and paths, we show in \cite{ZhoL20} the first known condition (that can be checked {\it a priori}) for OPF to have no spurious local optima.
The same condition was only known to guarantee exact relaxation before our work.

In the second example, we study the Low Rank Semidefinite Program (LRSDP) problem, which was known to have weakly exact relaxation \cite{barvinok1995problems, pataki1998rank} and no spurious local optima \cite{burer2005local} in existing literatures.
Specifically, we show that part of the results proved in \cite{burer2005local} can also be proved by finding appropriate $V$ and $\{h_x\}$.
They exemplify the usage of \cref{thm:suff}, \cref{co:lo_is_go} and \cref{thm:weak_suff} in practice.

\subsection{Optimal Power Flow}
Consider a radial power network with an underlying connected directed graph $\G(\V,\E)$.
Let $\V:=\{0, 1, \cdots, N-1\}$ be the set of buses (i.e., nodes), and $\E\subseteq\V\times\V$ be the set of power lines (i.e., edges).
We will refer to a power line from bus $j$ to bus $k$ by $j\to k$ or $(j,k)$ interchangeably.
For each power line $(j,k)$, its series admittance is denoted by $y_{jk}\in\Complex$, and its series impedance is hence $z_{jk}:=y_{jk}^{-1}$.
Both the real and imaginary parts of $z_{jk}$ are assumed to be positive.

As we assume $\G$ is a tree, we can adopt the DistFlow Model \cite{Baran1989a, Baran1989b} to formulate power flow equations. For each bus $j$, let $V_j\in\Complex, s_j=p_j+\iu q_j\in\Complex$ denote its voltage and bus injection respectively.
For line $(j,k)$, let $S_{jk}$ and $I_{jk}\in\Complex$ denote the branch power flow and current from bus $j$ to $k$, both at the sending end.
Let $v_j:=|V_j|^2\in\Real$ and $\ell_{jk}:=|I_{jk}|^2\in\Real$.
We will denote the conjugate of a complex number $a$ by $a^\Hn$.

The power flow equations are:
\begin{subequations}
\begin{align}
v_j &= v_k+2{\rm Re}(z_{jk}S_{jk}^\Hn)-|z_{jk}|^2\ell_{jk}, & \forall (j,k) &\in\E \label{eq:pf2.a}\\
v_j &= \frac{|S_{jk}|^2}{\ell_{jk}}, & \forall (j,k) &\in\E \label{eq:pf2.b}\\
s_j&=\sum_{k:j\to k}S_{jk}-\sum_{i:i\to j}(S_{ij}-z_{ij}\ell_{ij}), & \forall j &\in\V \label{eq:pf2.c}.
\end{align}
\label{eq:pf2}
\end{subequations}

Given a cost function $f(s):\Complex^{N}\to \Real$, we are interested in the following OPF problem:
\begin{subequations}
\begin{eqnarray}
\underset{x=(s,v,\ell,S)}{\text{minimize}}  && f(s)
\label{eq:opf.a}\\
\text{\quad subject to}& &\eqref{eq:pf2}
\label{eq:opf.b}\\
& & \underline{v}_j \leq v_j \leq \overline{v}_j
\label{eq:opf.c}\\
& & \underline{s}_j \leq s_j \leq \overline{s}_j
\label{eq:opf.d}\\
& &  \ell_{jk} \leq \overline{\ell}_{jk}
\label{eq:opf.e}
\end{eqnarray}
\label{eq:opf}
\end{subequations}
All the inequalities for complex numbers in this section are enforced for both the real and imaginary parts.

\begin{definition}
A function $g:\Real\to\Real$ is \emph{strongly increasing} if there exists real $c>0$ such that for any $a> b$, we have
\begin{align*}
g(a)-g(b)\geq c(a-b).
\end{align*}
\end{definition}
We now make the following assumptions on OPF:
\begin{enumerate}[label=(\roman*)]
\item The underlying graph $\G$ is a tree.
\item The cost function $f$ is convex, and is strongly increasing in ${\rm Re}(s_j)$ (or ${\rm Im}(s_j)$) for each $j\in\V$ and non-decreasing in ${\rm Im}(s_j)$ (or ${\rm Re}(s_j)$).
\item The problem \eqref{eq:opf} is feasible.
\item The line current limit satisfies $\overline{\ell}_{jk}\leq\underline{v}_j|y_{jk}|^2$.
\end{enumerate}
Assumption (i) is generally true for distribution networks and assumption (iii) is typically mild.
As for (ii), $f$ is commonly assumed to be convex and increasing in ${\rm Re}(s_j)$ and ${\rm Im}(s_j)$ in the literature (e.g., \cite{zhang2013geometry, gan2015exact}).
Assumption (ii) is only slightly stronger since one could always perturb any increasing function by an arbitrarily small linear term to achieve strong monotonicity. 
Assumption (iv) is not common in the literature but is also mild because of the following reason.
Typically $V_j = (1+\epsilon_j)e^{\iu \theta_j}$ in per unit where $\epsilon\in [-0.1, 0.1]$ 
and the angle difference $\theta_{jk} := \theta_j- \theta_k$ between two neighboring buses $j, k$ typically 
has a small magnitude.
Thus the maximum value of $|V_j-V_k|^2 = |(1+\epsilon_j)e^{\iu \theta_{jk}} - (1+\epsilon_k)|^2$, which is equivalent to $\overline{\ell}_{jk}/|y_{jk}|^2$, should be much smaller than $\underline{v}_j$ which is $\approx 1$ per unit.

Problem \eqref{eq:opf} is non-convex, as constraint \eqref{eq:pf2.b} is not convex.
Denote by $\X$ the set of $(s,v,\ell,S)$ that satisfy \eqref{eq:opf.b}-\eqref{eq:opf.e}, so \eqref{eq:opf} is in the form of \eqref{eq:opt}.
We can relax \eqref{eq:opf} by convexifying \eqref{eq:pf2.b} into a second-order cone \cite{farivar2013branchI}:
\begin{subequations}
\begin{eqnarray}
\underset{x=(s,v,\ell,S)}{\text{minimize}}  && f(s)
\label{eq:opfR.a}\\
\text{\quad subject to}& &\eqref{eq:pf2.a}, \eqref{eq:pf2.c}, \eqref{eq:opf.c}-\eqref{eq:opf.e}
\label{eq:opfR.b}\\
& & |S_{jk}|^2\leq v_j\ell_{jk}
\label{eq:opfR.c}
\end{eqnarray}
\label{eq:opfR}
\end{subequations}

One can similarly regard $\XR$ as the set of $(s,v,\ell,S)$ that satisfy \eqref{eq:opfR.b}, \eqref{eq:opfR.c}.
It is proved in \cite{farivar2013branchI} that if $\underline{s}_j=-\infty-\iu\infty$ for all $j\in\V$,
then \eqref{eq:opfR} is exact, meaning any optimal solution of \eqref{eq:opfR} is also feasible and hence globally optimal for \eqref{eq:opf}.
Now we show that the same condition also guarantees that any local optimum of \eqref{eq:opf} is also globally optimal.
This implies that a local search algorithm such as the primal-dual interior point method can produce a global optimum as long as it  converges.

\begin{theorem}\label{thm:opf}
If $\underline{s}_j=-\infty-\iu\infty$ for all $j\in\V$, then any local optimum of \eqref{eq:opf} is a global optimum.
\end{theorem}

\begin{proof}
Our strategy is to construct appropriate $V$ and $\{h_x\}$ and then prove such construction satsify both Condition (C) and Condition (C').
Let
\begin{align}\label{eq:V}
V(x):=\sum_{(j,k)\in\E}v_j\ell_{jk}-|S_{jk}|^2.
\end{align}
Clearly, $V$ is a valid Lyapunov-like function satisfying \cref{def:Lyapunov}.

For each $x=(s,v,\ell,S)\in\XR\setminus\X$, let $\M$ be the set of $(j,k)\in\E$ such that $|S_{jk}|^2< v_j\ell_{jk}$.
For $(j,k)\in\M$, 
the quadratic function
\begin{align*}
\phi_{jk}(a):=\frac{|z_{jk}|^2}{4}a^2+\big(v_j-{\rm Re}(z_{jk}S_{jk}^\Hn)\big)a + |S_{jk}|^2 - v_j\ell_{jk}
\end{align*}
must have a unique positive root as $\phi_{jk}(0)<0$.
We define $\Delta_{jk}$ to be this positive root if $(j,k)\in\M$ and $0$ otherwise.

Assumption (iv) implies $\ell_{jk}\leq v_j|y_{jk}|^2$, and therefore
\begin{align*}
&\ v_j-{\rm Re}(z_{jk}S_{jk}^\Hn)\ \geq \ v_j-|z_{jk}||S_{jk}|\\
\geq &\ v_j-|z_{jk}|\sqrt{v_j\ell_{jk}}
\ \geq \ v_j-|z_{jk}|\sqrt{v_j^2|y_{jk}|^2} \ \ = \ \ 0.
\end{align*}
It further implies $\phi_{jk}(a)$ is strictly increasing for $a\in[0,\Delta_{jk}]$.

Now consider the path $h_x(t):=(\hs(t),\hv(t),\hl(t),\hS(t))$ for $t\in[0,1]$, where
\begin{subequations}
\begin{align}
\hs_j(t)&=s_j-\frac{t}{2}\sum_{i:i\to j}z_{ij}\Delta_{ij}-\frac{t}{2}\sum_{k:j\to k}z_{jk}\Delta_{jk},\\
\hv_j(t)&=v_j,\\
\hl_{jk}(t)&=\ell_{jk}-t\Delta_{jk},\\
\hS_{jk}(t)&=S_{jk}-\frac{t}{2}z_{jk}\Delta_{jk}.
\end{align}
\label{eq:path}
\end{subequations}

Clearly we have $h_x(0)= x$.
It can be easily checked that $h_x(t)$ is feasible for \eqref{eq:opfR} for $t\in[0,1]$ 
and $h_x(1)$ is feasible for \eqref{eq:opf} (see \cite{ZhoL20}).
Therefore, $h_x$ is indeed $[0,1]\to\XR$ and $h_x(1)\in\X$. 

Since $z_{jk}>0$, both real and imaginary parts of $\hs_j(t)$ are strictly decreasing for $(j,k)\in\M$ and stay unchanged otherwise.
By assumption (ii), $f(\hs(t))$ is also strictly decreasing.
To show $V(h_x(t))$ is also decreasing, we notice that $V(h_x(t))$ equals to
\begin{align*}
&\sum_{(j,k)\in\E}\hv_j(t)\hl_{jk}(t)-|\hS_{jk}(t)|^2\\
=&\sum_{(j,k)\in\M^\co}v_j\ell_{jk}-|S_{jk}|^2 + \sum_{(j,k)\in\M}\hv_j(t)\hl_{jk}(t)-|\hS_{jk}(t)|^2\\
=&\sum_{(j,k)\in\M^\co}v_j\ell_{jk}-|S_{jk}|^2 - \sum_{(j,k)\in\M}\phi_{jk}(t\Delta_{jk}).
\end{align*}
As $\phi_{jk}(a)$ is strictly increasing for $a\in[0,\Delta_{jk}]$, we conclude that $V(h_x(t))$ is strictly decreasing for $t\in[0,1]$.

By \cref{cor:linear}, the set $\{h_x\}_{x\in\XR\setminus\X}$ is uniformly bounded and uniformly equicontinuous as all $h_x(t)$ are linear functions in $t$.
In summary, Condition (C) is satisfied.

Finally, we show Condition (C') also holds.
By assumption (ii), there exists some real $c>0$ independent of $x$ such that for any $0\leq a<b\leq 1$,
\begin{align*}
&f(\hs(a))-f(\hs(b)) \\
\geq &\ c\sum_{j\in\V}{\rm Re}(\hs_j(a)-\hs_j(b))+{\rm Im}(\hs_j(a)-\hs_j(b))\\
= & \ c\|\hs(a)-\hs(b)\|_{\rm m}
\end{align*}
where $\|\cdot\|_{\rm m}$ is defined as $\|{\bf a}\|_{\rm m}:=\sum_i|{\rm Re}(a_i)|+|{\rm Im}(a_i)|$ over the complex vector space.
It is easy to check $\|\cdot\|_{\rm m}$ is a valid norm.

On the other hand, by \eqref{eq:path} we have $\|\hv(a)-\hv(b)\|_{\rm m}\equiv0$ and
\begin{align*}
\|\hl(a)-\hl(b)\|_{\rm m}&\leq \frac{1}{\min\limits_{(j,k)\in\E}\{\|z_{jk}\|_{\rm m}\}}\|\hs(a)-\hs(b)\|_{\rm m},\\
\|\hS(a)-\hS(b)\|_{\rm m}&\leq \frac{1}{2}\|\hs(a)-\hs(b)\|_{\rm m}.
\end{align*}
Therefore,
\begin{align*}
\|h_x(a)-h_x(b)\|_{\rm m}
\leq \Big(\frac{3}{2}+\frac{1}{\min\limits_{(j,k)\in\E}\{\|z_{jk}\|_{\rm m}\}}\Big)\|\hs(a)-\hs(b)\|_{\rm m}
\end{align*}
and there exists $\hat{c}>0$ independent of $x, a, b$ such that 
\begin{align*}
f(\hs(a))-f(\hs(b))\geq \hat{c}\|h_x(a)-h_x(b)\|_{\rm m}.
\end{align*}
Therefore Condition (C') is also satisfied and by Theorem \ref{co:lo_is_go}, any local optimum of \eqref{eq:opf} is a global optimum.
\end{proof}

The results in this subsection only apply to radial networks, which serve as the underlying network for balanced distribution power systems.
For transmission systems and unbalanced distribution systems, networks are usually highly meshed.
It has been found that for most of meshed networks, both convex relaxation and local search algorithms can also yield the global optimum for most of testcases \cite{Jabr02,gopinath2020proving}.
Thus \cref{thm:nece} suggests that there may also exist similar Lyapunov-like function and paths for meshed networks.
Finding such Lyapunov-like function and paths would be an interesting future work to extend our results in this paper.

\subsection{Low Rank Semidefinite Program}
This subsection proves a known result in \cite{burer2005local} but using a different approach. 
Adopting the same notations as \cite{burer2005local}, we have the following problem.
\begin{subequations}
\begin{eqnarray}
\underset{X\geq 0}{\text{minimize}}  && \tr(CX)
\label{eq:lrsdp.a}\\
\text{\quad subject to}& &
\label{eq:lrsdp.b}
\tr(A_iX) = b_i,\quad i=1,\cdots,m\\
& & \rank(X)\leq r
\label{eq:lrsdp.c}
\end{eqnarray}
\label{eq:lrsdp}
\end{subequations}
Here, $C$, $A_i$, $X$ are all $n$-by-$n$ matrices.
We assume the problem is feasible and $\{X\geq 0~|~\eqref{eq:lrsdp.b}\}$ is compact.

\begin{theorem}\label{thm:lrsdp}
If $(r+1)(r+2)/2 > m+1$, then any local optimum of \eqref{eq:lrsdp} is either a global optimum or a pseudo local optimum.
\end{theorem}
Before proving \cref{thm:lrsdp}, we consider the convex relaxation of \eqref{eq:lrsdp} as
\begin{subequations}
\begin{eqnarray}
\underset{X\geq 0}{\text{minimize}}  && \tr(CX)
\label{eq:lrsdp_rlx.a}\\
\text{\quad subject to}& &
\tr(A_iX) = b_i,\quad i=1,\cdots,m
\end{eqnarray}
\label{eq:lrsdp_rlx}
\end{subequations}
As a side note, the results in \cite{barvinok1995problems, pataki1998rank} show that if $(r+1)(r+2)/2 > m$, then \eqref{eq:lrsdp_rlx} is weakly exact to \eqref{eq:lrsdp}.
While our theorem is the same as in \cite{burer2005local}, some insights to find $V$ and $\{h_X\}$ are also from the structures first raised in \cite{barvinok1995problems, pataki1998rank}.

\begin{proof}
Clearly, \eqref{eq:lrsdp} can be reformulated in the form of \eqref{eq:opt} by setting $f(X)=\tr(CX)$, $\X=\{X\geq 0~|~\eqref{eq:lrsdp.b}, \eqref{eq:lrsdp.c}\}$
and $\XR=\{X\geq 0~|~\eqref{eq:lrsdp.b}\}$.
Define $V$ as
\begin{align*}
V(X):=\sum_{i=r+1}^n \lambda_i(X)
\end{align*}
where $\lambda_i(X)$ is the $i$\textsuperscript{th} eigenvalue of $X$ (in decreasing order).
This function $V$ satisfies \cref{def:Lyapunov} and is concave.

For fixed $X\in\XR\setminus\X$, we denote $\rank(X)$ as $r_0>r$.
We first construct $r_0-r$ paths labeled as $h_1,h_2,\cdots,h_{r_0-r}$.
When we construct $h_i$, if $i>1$ then we assume path $h_{i-1}$ has already been constructed and let $X_{i-1}:=h_{i-1}(1)$.
We let $X_0=X$.
For $i\geq 1$, if $\rank(X_{i-1})\leq r_0-i$ then we let $h_i(t)\equiv X_{i-1}$ for $t\in[0,1]$.
Otherwise, we decompose $X_{i-1}$ as $U\Sigma U^\Hn$ where $\Sigma$ is a $k$-by-$k$ positive definite diagonal matrix with $k=\rank(X_{i-1})>r_0-i$.
The linear system
\begin{align}
\nonumber
\tr(CUYU^\Hn)=&0\\\label{eq:lin_sys}
\tr(A_iUYU^\Hn)=&0,~ i=1,\cdots,m
\end{align}
must have a non-zero solution for Hermitian matrix $Y\in\Complex^{k\times k}$.
To see this, we have $k\geq r_0-i+1\geq r+1$, and thus $k(k+1)/2\geq (r+1)(r+2)/2>m+1$.
As a result, \eqref{eq:lin_sys} has more unknown variables than equations.
We simply denote this non-zero solution as $Y$ and for any $\alpha\in\Real$, $\alpha Y$ is also a solution to \eqref{eq:lin_sys}.
The concavity of $V$ also implies that $V(U(\Sigma+\alpha Y)U^\Hn)$ is concave in $\alpha$ when $U$ and $\Sigma$ are fixed.
Since $\Sigma>0$, one of the following two scenarios must be true.
\begin{itemize}
\item $\exists a<0$ such that $V(U(\Sigma+\alpha Y)U^\Hn)$ is non-decreasing, $\rank(U(\Sigma+\alpha Y)U^\Hn) \leq k$ for $\alpha\in[a,0]$ and $\rank(U(\Sigma+a Y)U^\Hn) \leq k-1$.
\item $\exists b>0$ such that $V(U(\Sigma+\alpha Y)U^\Hn)$ is non-increasing, $\rank(U(\Sigma+\alpha Y)U^\Hn) \leq k$  for $\alpha\in[0,b]$ and $\rank(U(\Sigma+b Y)U^\Hn) \leq k-1$.
\end{itemize}
Without loss of generality, we suppose $V(U(\Sigma+\alpha Y)U^\Hn)$ is non-increasing for $\alpha\in[0,b]$ (otherwise we take $-Y$ instead).
We then construct $h_i$ as $h_i(t)=U(\Sigma+t b Y)U^\Hn$ for $t\in[0,1]$.
By construction, $V(h_i(t))$ is non-increasing and $f(h_i(t))$ stays a constant.

Finally, we construct $h_X$ as the concatenation of paths $h_1,\cdots,h_{r_0-r}$.
That is to say,
\begin{align*}
h_X(t):=h_i((r_0-r)t-i+1) \text{ for } t\in\Big[\frac{i-1}{r_0-r},\frac{i}{r_0-r}\Big].
\end{align*}
It is easy to see $h_X$ is continuous and $h_X(0)=h_1(0)=X$.
To see $h_X(1)\in\X$, we prove that $\rank(X_i)\leq r_0-i$.
We first have $\rank(X_0)=\rank(X)=r_0$.
For $i\geq 1$, we have $\rank(X_i)=\rank(X_{i-1})$ if $\rank(X_{i-1})\leq r_0-i$ and $\rank(X_i)\leq \rank(X_{i-1})-1$ otherwise.  
By induction, we can prove $\rank(X_i)\leq r_0-i$ always holds.
As a result, $\rank(h_X(1))=\rank(h_{r_0-r}(1))\leq r$ and thus $h_X(1)\in\X$.
By construction, $h_i(t)$ never violates \eqref{eq:lrsdp.b} and thus is in $\XR$, so is $h_X(t)$ for all $t$.
Functions $V(h_i(t))$ and $f(h_i(t))$ being non-increasing implies that $V(h_X(t))$ and $f(h_X(t))$ are also non-increasing.
Therefore, \ref{C3} is satisfied.
By \cref{cor:linear} \ref{C2} also holds for $\{\overline{h_X}\}$.
It completes the proof (by \cref{thm:weak_suff}). 
\end{proof}
\begin{remark}
In \cite{burer2005local}, Theorem 3.4 claims that any local optimum of \eqref{eq:lrsdp} should also be globally optimal, unless it is harbored in some positive-dimensional face of SDP.
The result in our paper further asserts that if it is indeed harbored in such a face, then there must be some point on the edge of this face whose cost can be further reduced in its neighborhood
(i.e., the local optimum is in the same situation as point $c$ rather than $d$ as in \cref{fig:lo}).
\end{remark}
\section{Conclusion and Discussion}
\begin{table}[htp]
\caption{Sufficient and necessary conditions}
\begin{center}
\begin{tabular}{|c|c|c|}
\hline
Condition & Relaxation exactness & Local optimality\\
\Xhline{3\arrayrulewidth}
\multicolumn{3}{|l|}{Sufficient conditions: $\Rightarrow$}\\
\hline
\ref{C1}, \ref{C2} & Strong exactness & l.o. is p.l.o. or g.o. \\
\hline
\ref{C3}, \ref{C2} & Weak exactness & l.o. is p.l.o. or g.o. \\
\hline
(C') & Strong exactness & l.o. is g.o. \\
\Xhline{3\arrayrulewidth}
\multicolumn{3}{|l|}{Necessary condition: $\Leftarrow$}\\
\hline
\ref{C1}, \ref{C2} & Strong exactness & l.o. is g.o. \\
\hline
\end{tabular}
\end{center}
\label{tb:suff_nece}
\end{table}

\cref{tb:suff_nece} summaries both sufficient and necessary conditions for non-convex problem \eqref{eq:opt} to simultaneously have exact (weak or strong) relaxation and no spurious local optima (allowing or not allowing pseudo local optima).
The necessary condition relies on \cref{as:semianalytic}, which is usually true for real-world problems.
Those results provide a new perspective to certify a non-convex problem is computationally easy to solve.
Furthermore, whenever the problem is indeed computationally easy, the certificates (Lyapunov-like functions and paths) are guaranteed to exist.
We also provide a hierarchical framework which shows how such certificates for a complicated problem can be constructed from primitive problems.
Our results have been applied to OPF and LRSDP problems.

Based on the examples shown in \cref{sec:application}, a natural way to apply this approach is to first look at existing results on exact relaxation, and then construct $V$ and $\{h_x\}$ according to the hidden structure underlying the exactness.
Once $V$ and $\{h_x\}$ are appropriately constructed, our result can help extend existing results on relaxation exactness to new results on local optimality.

Compared to some existing techniques to study local optimality, our results do not require differentiating or analyzing the curvature of feasible sets.
It allows the feasible sets to incorporate more complicated and possibly non-convex constraints. Those non-convex constraints are common for problems arising in cyber physical systems which are generally governed by physical laws.   

\appendices

\ifCLASSOPTIONcaptionsoff
  \newpage
\fi



%

\bibliographystyle{IEEEtran}
\bibliography{my-bibliography,PowerRef-201202}

\begin{thebibliography}{10}
\providecommand{\url}[1]{#1}
\csname url@samestyle\endcsname
\providecommand{\newblock}{\relax}
\providecommand{\bibinfo}[2]{#2}
\providecommand{\BIBentrySTDinterwordspacing}{\spaceskip=0pt\relax}
\providecommand{\BIBentryALTinterwordstretchfactor}{4}
\providecommand{\BIBentryALTinterwordspacing}{\spaceskip=\fontdimen2\font plus
\BIBentryALTinterwordstretchfactor\fontdimen3\font minus
  \fontdimen4\font\relax}
\providecommand{\BIBforeignlanguage}[2]{{%
\expandafter\ifx\csname l@#1\endcsname\relax
\typeout{** WARNING: IEEEtran.bst: No hyphenation pattern has been}%
\typeout{** loaded for the language `#1'. Using the pattern for}%
\typeout{** the default language instead.}%
\else
\language=\csname l@#1\endcsname
\fi
#2}}
\providecommand{\BIBdecl}{\relax}
\BIBdecl

\bibitem{ZhoL20}
F.~Zhou and S.~H. Low, ``A sufficient condition for local optima to be globally
  optimal,'' in \emph{To appear in Proc. of the 2020 Conference on Decision and
  Control}.\hskip 1em plus 0.5em minus 0.4em\relax IEEE, 2020.

\bibitem{candes2009exact}
E.~J. Cand{\`e}s and B.~Recht, ``Exact matrix completion via convex
  optimization,'' \emph{Foundations of Computational mathematics}, vol.~9,
  no.~6, p. 717, 2009.

\bibitem{candes2010power}
E.~J. Cand{\`e}s and T.~Tao, ``The power of convex relaxation: Near-optimal
  matrix completion,'' \emph{IEEE Transactions on Information Theory}, vol.~56,
  no.~5, pp. 2053--2080, 2010.

\bibitem{ge2016matrix}
R.~Ge, J.~D. Lee, and T.~Ma, ``Matrix completion has no spurious local
  minimum,'' in \emph{Advances in Neural Information Processing Systems}, 2016,
  pp. 2973--2981.

\bibitem{barvinok1995problems}
A.~I. Barvinok, ``Problems of distance geometry and convex properties of
  quadratic maps,'' \emph{Discrete \& Computational Geometry}, vol.~13, no.~2,
  pp. 189--202, 1995.

\bibitem{pataki1998rank}
G.~Pataki, ``On the rank of extreme matrices in semidefinite programs and the
  multiplicity of optimal eigenvalues,'' \emph{Mathematics of operations
  research}, vol.~23, no.~2, pp. 339--358, 1998.

\bibitem{burer2005local}
S.~Burer and R.~D. Monteiro, ``Local minima and convergence in low-rank
  semidefinite programming,'' \emph{Mathematical Programming}, vol. 103, no.~3,
  pp. 427--444, 2005.

\bibitem{farivar2013branchI}
M.~Farivar and S.~H. Low, ``Branch flow model: Relaxations and
  convexification--part {I},'' \emph{IEEE Transactions on Power Systems},
  vol.~28, no.~3, pp. 2554--2564, 2013.

\bibitem{gan2015exact}
L.~Gan, N.~Li, U.~Topcu, and S.~H. Low, ``Exact convex relaxation of optimal
  power flow in radial networks,'' \emph{IEEE Transactions on Automatic
  Control}, vol.~60, no.~1, pp. 72--87, 2015.

\bibitem{lavaei2012zero}
J.~Lavaei and S.~H. Low, ``Zero duality gap in optimal power flow problem,''
  \emph{IEEE Transactions on Power Systems}, vol.~27, no.~1, pp. 92--107, 2012.

\bibitem{jalden2003semidefinite}
J.~Jald{\'e}n, C.~Martin, and B.~Ottersten, ``Semidefinite programming for
  detection in linear systems-optimality conditions and space-time decoding,''
  in \emph{2003 IEEE International Conference on Acoustics, Speech, and Signal
  Processing, 2003. Proceedings.(ICASSP'03).}, vol.~4.\hskip 1em plus 0.5em
  minus 0.4em\relax IEEE, 2003, pp. IV--9.

\bibitem{lu2019tightness}
C.~Lu, Y.-F. Liu, W.-Q. Zhang, and S.~Zhang, ``Tightness of a new and enhanced
  semidefinite relaxation for mimo detection,'' \emph{SIAM Journal on
  Optimization}, vol.~29, no.~1, pp. 719--742, 2019.

\bibitem{li2015sufficient}
Z.~Li, Q.~Guo, H.~Sun, and J.~Wang, ``Sufficient conditions for exact
  relaxation of complementarity constraints for storage-concerned economic
  dispatch,'' \emph{IEEE Transactions on Power Systems}, vol.~31, no.~2, pp.
  1653--1654, 2015.

\bibitem{sun2015nonconvex}
J.~Sun, Q.~Qu, and J.~Wright, ``When are nonconvex problems not scary?''
  \emph{arXiv preprint arXiv:1510.06096}, 2015.

\bibitem{ge2017no}
R.~Ge, C.~Jin, and Y.~Zheng, ``No spurious local minima in nonconvex low rank
  problems: A unified geometric analysis,'' in \emph{Proceedings of the 34th
  International Conference on Machine Learning-Volume 70}.\hskip 1em plus 0.5em
  minus 0.4em\relax JMLR. org, 2017, pp. 1233--1242.

\bibitem{sun2016complete}
J.~Sun, Q.~Qu, and J.~Wright, ``Complete dictionary recovery over the sphere
  {I}: Overview and the geometric picture,'' \emph{IEEE Transactions on
  Information Theory}, vol.~63, no.~2, pp. 853--884, 2016.

\bibitem{boumal2016nonconvex}
N.~Boumal, ``Nonconvex phase synchronization,'' \emph{SIAM Journal on
  Optimization}, vol.~26, no.~4, pp. 2355--2377, 2016.

\bibitem{arora2015simple}
S.~Arora, R.~Ge, T.~Ma, and A.~Moitra, ``Simple, efficient, and neural
  algorithms for sparse coding,'' \emph{Proceedings of Machine Learning
  Research}, vol.~40, January 2015.

\bibitem{carpentier1962contribution}
J.~Carpentier, ``Contribution to the economic dispatch problem,''
  \emph{Bulletin de la Societe Francoise des Electriciens}, vol.~3, no.~8, pp.
  431--447, 1962.

\bibitem{Verma2009}
A.~Verma, ``Power grid security analysis : An optimization approach,'' Ph.D.
  dissertation, Columbia University, 2009.

\bibitem{lehmann2016ac}
K.~Lehmann, A.~Grastien, and P.~Van~Hentenryck, ``{AC}-feasibility on tree
  networks is {NP}-hard,'' \emph{IEEE Transactions on Power Systems}, vol.~31,
  no.~1, pp. 798--801, 2016.

\bibitem{momoh1999reviewa}
J.~A. Momoh, R.~Adapa, and M.~El-Hawary, ``A review of selected optimal power
  flow literature to 1993. {I}. nonlinear and quadratic programming
  approaches,'' \emph{IEEE transactions on power systems}, vol.~14, no.~1, pp.
  96--104, 1999.

\bibitem{momoh1999reviewb}
J.~A. Momoh, M.~El-Hawary, and R.~Adapa, ``A review of selected optimal power
  flow literature to 1993. {II}. newton, linear programming and interior point
  methods,'' \emph{IEEE Transactions on Power Systems}, vol.~14, no.~1, pp.
  105--111, 1999.

\bibitem{jabr2002primal}
R.~A. Jabr, A.~H. Coonick, and B.~J. Cory, ``A primal-dual interior point
  method for optimal power flow dispatching,'' \emph{IEEE Transactions on Power
  Systems}, vol.~17, no.~3, pp. 654--662, 2002.

\bibitem{jabr2006radial}
R.~A. Jabr, ``Radial distribution load flow using conic programming,''
  \emph{IEEE transactions on power systems}, vol.~21, no.~3, pp. 1458--1459,
  2006.

\bibitem{Bai2008}
X.~Bai, H.~Wei, K.~Fujisawa, and Y.~Wang, ``Semidefinite programming for
  optimal power flow problems,'' \emph{Int'l J. of Electrical Power \& Energy
  Systems}, vol.~30, no. 6-7, pp. 383--392, 2008.

\bibitem{bose2015quadratically}
S.~Bose, D.~F. Gayme, K.~M. Chandy, and S.~H. Low, ``Quadratically constrained
  quadratic programs on acyclic graphs with application to power flow,''
  \emph{IEEE Transactions on Control of Network Systems}, vol.~2, no.~3, pp.
  278--287, 2015.

\bibitem{low2014convex}
S.~H. Low, ``Convex relaxation of optimal power flow--part {II}: Exactness,''
  \emph{IEEE Transactions on Control of Network Systems}, vol.~1, no.~2, pp.
  177--189, 2014.

\bibitem{molzahn2019survey}
D.~K. Molzahn and I.~A. Hiskens, ``A survey of relaxations and approximations
  of the power flow equations,'' \emph{Foundations and Trends{\textregistered}
  in Electric Energy Systems}, vol.~4, no. 1-2, pp. 1--221, 2019.

\bibitem{toponogov2006differential}
V.~A. Toponogov, \emph{Differential geometry of curves and surfaces}.\hskip 1em
  plus 0.5em minus 0.4em\relax Springer, 2006.

\bibitem{bierstone1988semianalytic}
E.~Bierstone and P.~D. Milman, ``Semianalytic and subanalytic sets,''
  \emph{Publications Math{\'e}matiques de l'Institut des Hautes {\'E}tudes
  Scientifiques}, vol.~67, no.~1, pp. 5--42, 1988.

\bibitem{bierstone1980differentiable}
E.~Bierstone, ``Differentiable functions,'' \emph{Boletim da Sociedade
  Brasileira de Matem{\'a}tica-Bulletin/Brazilian Mathematical Society},
  vol.~11, no.~2, pp. 139--189, 1980.

\bibitem{hardt1983some}
R.~Hardt, ``Some analytic bounds for subanalytic sets,'' \emph{Differential
  geometric control theory, Progress in Math}, vol.~27, pp. 259--267, 1983.

\bibitem{gabrielov1968projections}
A.~M. Gabri{\`e}lov, ``Projections of semi-analytic sets,'' \emph{Functional
  Analysis and its applications}, vol.~2, no.~4, pp. 282--291, 1968.

\bibitem{lojasiewicz1995semi}
S.~{\L}ojasiewicz, ``On semi-analytic and subanalytic geometry,'' \emph{Banach
  Center Publications}, vol.~34, no.~1, pp. 89--104, 1995.

\bibitem{Baran1989a}
M.~E. Baran and F.~F. Wu, ``{Optimal Capacitor Placement on radial distribution
  systems},'' \emph{IEEE Trans. Power Delivery}, vol.~4, no.~1, pp. 725--734,
  1989.

\bibitem{Baran1989b}
------, ``{Optimal Sizing of Capacitors Placed on A Radial Distribution
  System},'' \emph{IEEE Trans. Power Delivery}, vol.~4, no.~1, pp. 735--743,
  1989.

\bibitem{zhang2013geometry}
B.~Zhang and D.~Tse, ``Geometry of injection regions of power networks,''
  \emph{IEEE Transactions on Power Systems}, vol.~28, no.~2, pp. 788--797,
  2013.

\bibitem{Jabr02}
R.~Jabr, A.~Coonick, and B.~Cory, ``A primal-dual interior point method for
  optimal power flow dispatching,'' \emph{IEEE Trans. on Power Systems},
  vol.~17, no.~3, pp. 654--662, 2002.

\bibitem{gopinath2020proving}
S.~Gopinath, H.~Hijazi, T.~Wei{\ss}er, H.~Nagarajan, M.~Yetkin, K.~Sundar, and
  R.~Bent, ``Proving global optimality of acopf solutions,'' \emph{Electric
  Power Systems Research}, vol. 189, p. 106688, 2020.

\end{thebibliography}

%
%
%
%
%




\end{document}